\documentclass[11pt]{article}

\usepackage{amsmath,amssymb,amsthm,color,graphicx}
\usepackage{epsfig}
\usepackage{subfig}
\usepackage{comment}
\usepackage{epstopdf}
\usepackage{url}
\usepackage{hyperref}
\usepackage{cleveref}
\usepackage{enumitem}
\usepackage{thmtools,thm-restate}

\newtheorem{theorem}{Theorem}
\newtheorem{lemma}{Lemma}

\newtheorem{proposition}{Proposition}
\newtheorem{conjecture}{Conjecture}
\newtheorem{corollary}{Corollary}
\newtheorem{observation}{Observation}

\theoremstyle{definition}
\newtheorem{definition}{Definition}

\newcommand{\ex}{\mbox{ex}}

\newcommand{\ignore}[1]{}

\def\clap#1{\hbox to 0pt{\hss#1\hss}}

\makeatletter
  \def\moverlay{\mathpalette\mov@rlay}
  \def\mov@rlay#1#2{\leavevmode\vtop{%
    \baselineskip\z@skip \lineskiplimit-\maxdimen
    \ialign{\hfil$#1##$\hfil\cr#2\crcr}}}
\makeatother

\def\L{{\cal L}}
\def\C{{\cal C}}

\def\S{{\cal S}}

\newcommand{\remove}[1]{}

\newcommand{\E}[1]{\mathbb{E}(#1)}

\begin{document}
\title{On the maximum number of tangencies among $1$-intersecting curves}

	\author{
	Eyal Ackerman\thanks{Department of Mathematics, Physics and Computer Science,
		University of Haifa at Oranim, 	Tivon 36006, Israel. \texttt{ackerman@math.haifa.ac.il}}\and
	Bal\'azs Keszegh\thanks{HUN-REN Alfréd Rényi Institute of Mathematics and ELTE Eötvös Loránd University, Budapest, Hungary. 
		Research supported by the National Research, Development and Innovation Office -- NKFIH under the grant K 132696 and by the ERC Advanced Grant ``ERMiD''. This research has been implemented with the support provided by the Ministry of Innovation and Technology of Hungary from the National Research, Development and Innovation Fund, financed under the  ELTE TKP 2021-NKTA-62 funding scheme.}
}
\maketitle

\begin{abstract}
According to a conjecture of Pach, there are $O(n)$ tangent pairs among any family of $n$ Jordan arcs in which every pair of arcs has precisely one common point and no three arcs share a common point.
This conjecture was proved for two special cases, however, for the general case the currently best upper bound is only $O(n^{7/4})$. This is also the best known bound on the number of tangencies in the relaxed case where every pair of arcs has \emph{at most} one common point.
We improve the bounds for the latter and former cases to $O(n^{5/3})$ and $O(n^{3/2})$, respectively.
We also consider a few other variants of these questions, for example, we show that if the arcs are \emph{$x$-monotone}, each pair intersects at most once and their left endpoints lie on a common vertical line, then the maximum number of tangencies is $\Theta(n^{4/3})$.
Without this last condition the number of tangencies is $O(n^{4/3}(\log n)^{1/3})$,  improving a previous bound of Pach and Sharir.
Along the way we prove a graph-theoretic theorem which extends a result of Erd\H{o}s and Simonovits and may be of independent interest.
\end{abstract}

\section{Introduction}

Arrangements of curves play an important role in Computational and Combinatorial Geometry with applications in domains such as robotics,
computer graphics, computer vision, and combinatorial optimization~\cite{PachSharir2009_CGAA}.
Many interesting problems in these areas can be phrased as
determining the maximum number of \emph{tangencies} among curves under various restrictions.
For example, occasionally, analyzing the complexity of an arrangement of curves boils down to estimating the number of faces of size two~\cite{lenses}.
Such faces can often be perturbed into tangency points.
Another example is Erd\H{o}s's famous \emph{Unit Distance Problem} which asks for the maximum number of pairs of points that are at unit distance from each other among $n$ distinct points in the plane. 
This is equivalent to asking for the maximum number of tangencies among $n$ unit circles in the plane.

Throughout this paper we consider planar curves, that is, Jordan arcs,
and assume that every two curves intersect in a finite number of points. 
A family of curves is \emph{$k$-intersecting} if every two curves in it intersect in at most $k$ points.
Two curves are said to be \emph{tangent} or \emph{touching} if they have precisely one common point and this point is not a crossing point.\footnote{Here we follow the definition in, e.g.,~\cite{PRT18}. In other papers, e.g.,~\cite{Ellenberg2016New}, two curves may cross and also touch each other (at different points). However, for $1$-intersecting curves the two definitions coincide.}
We will mainly consider the maximum number of tangencies among families of $1$-intersecting curves. 
Such families are sometimes called \emph{pseudo-segments}.\footnote{However, the term `pseudo-segments' is somewhat ill-suited in the context of tangencies, since straight-line segments can only be tangent at an endpoint.}
A family of curves is \emph{precisely} $1$-intersecting if every pair of curves has precisely one common point (either a crossing or a tangency point).
We will always assume that no three curves share a common point, for otherwise, it is possible to have one point at which every pair of curves is tangent.
Therefore, the number of tangent pairs is equal to the number of tangency points.

J\'anos Pach proposed the following nice conjecture.
\begin{conjecture}[Pach~\cite{pachpc}]\label{conj:precisely-1}
	Every family of $n$ precisely $1$-intersecting planar curves in which no three curves share a common point admits $O(n)$ tangencies.
\end{conjecture}

Gy{\"o}rgyi et al.~\cite{GYORGYI201829} proved this conjecture for the very special case where the endpoints of the curves lie within a constant number of cells of their arrangement.
Conjecture~\ref{conj:precisely-1} was also settled by the authors for \emph{$x$-monotone} curves~\cite{x-mon}.
Recall that a curve is \emph{$x$-monotone} (resp., \emph{bi-infinite} $x$-monotone) if it is the graph of some continuous function over some closed interval (resp., over $\mathbb{R}$).
For arbitrary families of (precisely) $1$-intersecting curves a non-trivial $O(n^{7/4})$ upper bound on the number of tangencies follows from a result of Keszegh and P{\'a}lv{\"o}lgyi~\cite{KP23}:

\begin{theorem}[~\cite{KP23}]\label{thm:KP}
	Every family of $n$ $k$-intersecting planar curves admits $O_k\left(n^{2-\frac{1}{k+3}}\right)$ tangencies.
\end{theorem}

We improve this bound for (precisely) $1$-intersecting families as follows.

\begin{restatable}{theorem}{OneIntThm}\label{thm:1-int}
	Every family of $n$ $1$-intersecting planar curves admits $O(n^{5/3})$ tangencies.
\end{restatable}

\begin{restatable}{theorem}{PWOneIntThm}\label{thm:precisely-1-int}
	Every family of $n$ precisely $1$-intersecting planar curves admits $O(n^{3/2})$ tangencies.
\end{restatable}

Note that the maximum number of tangencies among $n$ ($x$-monotone) $1$-intersecting curves is $\Omega(n^{4/3})$.
Indeed, it is not hard to construct a family admitting that many tangencies based on the famous construction of Erd\H{o}s (see~\cite{pachbook}) of $n$ lines, $n$ points and $\Theta(n^{4/3})$ point-line incidences (see the proof of Theorem~\ref{thm:grounded-x-mon} below).

For $x$-monotone $1$-intersecting curves we have the following bound.
\begin{restatable}{theorem}{xmonThm}\label{thm:x-mon}
	Every family of $n$ $1$-intersecting $x$-monotone planar curves admits $O\left(n^{4/3}(\log n)^{1/3}\right)$ tangencies.
\end{restatable}

This bound almost matches the above-mentioned $\Omega(n^{4/3})$ lower bound and improves by an $O\left((\log n)^{1/3}\right)$ factor a previous bound that follows from a result of Pach and Sharir~\cite{Pach-Sharir}.
For so-called \emph{grounded} $x$-monotone $1$-intersecting curves we have an asymptotically tight bound.

\begin{restatable}{theorem}{groundedXmonThm}\label{thm:grounded-x-mon}
	The maximum number of tangencies among $n$ $x$-monotone $1$-intersecting planar curves whose left endpoints lie on a common vertical line is $\Theta(n^{4/3})$.
\end{restatable}

The proofs of the upper bounds in Theorems~\ref{thm:x-mon} and~\ref{thm:grounded-x-mon} rely on bounds for \emph{bi-infinite} $x$-monotone curves.
\begin{theorem}\label{thm:bi-infinite}
	The maximum number of tangencies among $n$ bi-infinite $x$-monotone $1$-intersecting curves is $\Theta(n\log n)$. In case of precisely $1$-intersecting curves the maximum number of tangencies is $n-1$.
\end{theorem}

The bound for bi-infinite $x$-monotone $1$-intersecting curves was already proved by Pach and Sharir~\cite{Pach-Sharir} using different and less elementary arguments.
The following table summarizes all the above-mentioned results. 



\begin{table}[h]
\centering
	\begin{tabular}{|c|c|c|}
		\hline
		&  $x$-monotone & arbitrary \\ \hline
		bi-infinite $1$-intersecting & $\Theta(n \log n)$ \cite{Pach-Sharir} (\autoref{sec:infinite-mon}) & --- \\ \hline
		bi-infinite precisely $1$-intersecting & $n-1$ (\autoref{sec:infinite-mon}) & --- \\ \hline
		grounded $1$-intersecting & $\Theta(n^{4/3})$ (\autoref{sec:grounded-mon}) & $\Omega(n^{4/3})$, $O(n^{3/2})$ (\autoref{sec:1-int}) \\ \hline
		$1$-intersecting & $\Omega(n^{4/3})$ \cite{Pach-Sharir}, $O(n^{4/3}(\log n)^{1/3})$ (\autoref{sec:mon}) & $\Omega(n^{4/3})$, $O(n^{5/3})$ (\autoref{sec:1-int})\\ \hline
		precisely $1$-intersecting & 	$\Theta(n)$ \cite{x-mon} & $\Omega(n)$, $O(n^{3/2})$ (\autoref{sec:1-int}) \\ \hline
	\end{tabular} 
	\caption{Bounds on the maximum number of tangencies among $n$ $1$-intersecting curves.}
\end{table}

\paragraph{A graph-theoretic theorem.}
For a graph $H$ let $\ex(n,H)$ denote the maximum number of edges in an $n$-vertex \emph{$H$-free} graph, that is, a graph which does not contain $H$ as a subgraph.
For a bipartite graph $H=\left(A\cup B,E\right)$ we denote by $H^+$ the graph we get by adding to $H$ two new adjacent vertices $a'$ and $b'$ and connecting $a'$ to every vertex in $B$ and $b'$ to every vertex in $A$.
That is, 
$H^+ = \left((A\cup \{a'\}) \cup (B \cup \{b'\}), E \cup \{(a',b')\} \cup \{(a',b) \mid b \in B\} \cup \{(a,b') \mid a \in A\} \right)$.
Erd\H{o}s and Simonovits~\cite{ES70} proved the following result\footnote{We present a simplified version of their result from which the original version can be easily deduced.} when studying $\ex(n,Q_3)$, where $Q_3$ is the graph with $8$ vertices and $12$ edges that corresponds to a three-dimensional cube.

\begin{theorem}[{\cite[Theorem 2]{ES70}}]\label{thm:erdos-simonovits}
	Let $H$ be a bipartite graph with at least one edge such that $\ex(n,H)=O(n^{2 -\alpha})$ for some constant $0\le\alpha \le 1$.
	Then $\ex(n,H^+) = O\left(n^{2-\frac{\alpha}{\alpha+1}}\right)$.
\end{theorem}
	
Observe that if $G=(A\cup B,E)$ is a bipartite $H^+$-free graph, then for every adjacent vertices $a \in A$ and $b \in B$ it holds that $G[(N_G(a)\setminus \{b\}) \cup (N_G(b)\setminus \{a\})]$ --- the subgraph of $G$ induced by their \emph{open bineighborhood} --- is $H$-free.
Roughly speaking, if for every adjacent vertices in $G$ the subgraph induced by their open bineighborhoods is not too dense because it is $H$-free, then it follows that $G$ is also not too dense.
We extend and refine Theorem~\ref{thm:erdos-simonovits} by showing that it holds whenever we have sparse subgraphs induced by open sub-bineighborhoods of pairs of (adjacent) vertices, be it for $H$-freeness or some other reason.
To state the result we need the following definition:
for a graph $G$ and a nonnegative and nondecreasing function $f:\mathbb{R} \rightarrow \mathbb{R}$ we say that two vertices $u$ and $v$ have \emph{$f$-sparse sub-bineighborhoods}
if for every two disjoint subsets $U \subseteq N_G(u) \setminus \{v\}$ and $V \subseteq N_G(v) \setminus \{u\}$ it holds that $|E(G(U,V))| \le f(|U \cup V|)$, where $G(U,V)$ is the bipartite subgraph $(U \cup V, \{(u,v) \in E(G) \mid u\in U \textrm{ and } v\in V\})$.
For the proof of Theorem~\ref{thm:1-int} we use the following theorem which extends Theorem~\ref{thm:erdos-simonovits} 
and may be of independent interest.

\begin{restatable}{theorem}{KSTub}
	\label{thm:spare-neighborhoods-ub}
	Let $G$ be a graph and let $f(n)=O(n^{2-\alpha})$ for some constant $0 \le \alpha <2$.
	If every pair of vertices in $G$ has $f$-sparse sub-bineighborhoods, then $|E(G)|  = O\left(|V(G)|^{2-\frac{\alpha}{\alpha+1}}\right)$.
	Moreover, if $\alpha \le 1$, then it is enough that every pair of adjacent vertices has $f$-sparse sub-bineighborhoods for this bound to hold.
\end{restatable}	

Clearly, Theorem~\ref{thm:spare-neighborhoods-ub} implies Theorem~\ref{thm:erdos-simonovits}. 
Note that from Theorem~\ref{thm:erdos-simonovits} one can conclude by induction on $t$ the classical K\H{o}v\'ari-S\'os-Tur\'an Theorem~\cite{Kovari} stating that an $n$-vertex $K_{t,t}$-free graph has $O_t\left(n^{2-\frac{1}{t}}\right)$ edges.
It is a major open question in extremal combinatorics whether this bound is tight with positive answers only known for $t=2$ and $t=3$.
Thus, the upper bound in Theorem~\ref{thm:erdos-simonovits} is tight for $\alpha=1$ and $\alpha=\frac{1}{2}$, and, assuming that the bound of the K\H{o}v\'ari-S\'os-Tur\'an Theorem is tight, for every $\alpha=\frac{1}{t}$ where $t$ is a positive integer.
Remarkably, for Theorem~\ref{thm:spare-neighborhoods-ub} we do have a matching lower bound for every $0 \le \alpha \le 1$.

\begin{restatable}{theorem}{KSTlb}
	\label{thm:spare-neighborhoods-lb}
	For every $0 \le \alpha \le 1$ there is a bipartite graph $G=(V,E)$ with $|E|=\Omega\left(|V|^{2-\frac{\alpha}{\alpha+1}}\right)$ edges such that every pair of vertices in $G$ has $f$-sparse sub-bineighborhoods where $f(n)=O(n^{2-\alpha})$.
\end{restatable}

\paragraph{Related work.}
Agarwal, Nevo, Pach, Pinchasi, Sharir and Smorodinsky~\cite{lenses} proved that
any family of $n$ pairwise intersecting $x$-monotone bi-infinite $2$-intersecting  curves admits at most $2n-4$ tangencies.
Families of $2$-intersecting closed curves are called \emph{pseudo-circles}.
Erd\H{o}s and Gr\"unbaum~\cite{EG73} (see also~\cite[\textsection 
7.1, Problem 14]{BMP05}) asked for the maximum number of tangencies among $n$ pseudo-circles.
An $\Omega(n^{4/3})$ lower bound for this problem follows  from the many point-line incidences construction.
The best upper bound is $O(n^{3/2})$ by a recent result of Janzer, Janzer, Methuku and Tardos~\cite{janzer2025}.
As for pairwise intersecting pseudo-circles, a recent work with Dam\'{a}sdi, Pinchasi and Raffay~\cite{psuedocircles-socg} settled an old conjecture of Gr\"unbaum~\cite{GRUNB72} by showing that the maximum number of tangencies is $2n-2$ in this case. 
A worse linear upper bound was proved in~\cite{lenses}.

Ellenberg, Solymosi and Zahl~\cite{Ellenberg2016New} proved that any family of $n$ plane algebraic curves of degree at most $d$ admits $O_d(n^{3/2})$ tangencies.\footnote{As mentioned before, their notion of tangency is a bit different and they actually bound the number of tangency points.}
If two curves may intersect an arbitrary number of times, then it is possible to have $\Omega(n^2)$ tangencies among $n$ curves.
Solving a conjecture of Richter and Thomassen it was shown by Pach, Rubin and Tardos~\cite{PRT18} that if every pair of $n$ closed curves intersects, then there are many more crossing points than tangency points.
Specifically, they showed that if $X$ denotes the crossing points and $T$ denotes the tangency points, then $|X| = \Omega(|T|(\log \log n)^{1/8})$ and conjectured that $|X| = \Omega(|T|\log n)$.

Another line of research studies the number of tangencies between two families each consisting of disjoint curves.
This problem was introduced by Pach, Suk and Treml~\cite{Treml} who attributed it to Ben-Dan and Pinchasi~\cite{Ben-Dan}.
An upper bound of $O(n^{3/2})$ follows from an observation of Ben-Dan and Pinchasi and from the result of Janzer et al.~\cite{janzer2025}, whereas a lower bound of $\Omega{(n^{4/3})}$ was shown by Keszegh and P{\'a}lv{\"o}lgyi~\cite{KP23} using $2$-intersecting curves.
For two families of $1$-intersecting curves a linear upper bound was proved by P{\'a}lv{\"o}lgyi et nos~\cite{Lshapes}.

For further related results on tangencies among curves and their applications see~\cite{lenses} and~\cite{PRT18} and the references therein.

\paragraph{Organization.}
After introducing some notations and collecting a few useful lemmas in Section~\ref{sec:tools}, we prove the results for arbitrary $1$-intersecting curves in Section~\ref{sec:1-int}.
Our results for $x$-monotone curves are proved in Section~\ref{sec:x-mon}.
Finally, in Section~\ref{sec:sparse-neighborhoods} we prove Theorems~\ref{thm:spare-neighborhoods-ub} and~\ref{thm:spare-neighborhoods-lb} and discuss their further applications and connections to related graph theoretic results.

\section{Notations and Tools}
\label{sec:tools}

Recall that we assume that every two curves intersect at finitely many points and that no three curves intersect at a single point.
It follows from the former assumption that we may also assume that the curves are drawn as polygonal chains.
Indeed, given a set of curves $\C$, consider the plane graph $G$ whose vertex set consists of the endpoints and intersection points of the curves in $\C$ and whose edge set consists of subcurves between consecutive vertices along the curves.
By F\'ary's Theorem~\cite{Fary} this planar graph can be embedded using straight-line edges and the same rotation system.
Replacing every curve $c \in \C$ with the polygonal chain that consists of the edges in this embedding that correspond to edges of $G$ that belong to $c$, we obtain a family of curves $\C'$ such that for every intersection (resp., tangency) point of two curves in $\C$ there is a unique intersection (resp., tangency) point between their corresponding curves in $\C'$, and vice versa.

For a set of curves $\C$ we denote by $G_\C$ the \emph{tangency graph} of $\C$, that is, the graph whose vertex set is $\C$ and whose edge set consists of the tangent pairs of curves.
By abuse of notation we sometimes interchange curves in $\C$ and their corresponding vertices in $G_\C$.

Let $c$ be a curve and let $p$ and $q$ be two points on $c$.
We denote by $c[p,q]$ the subcurve of $c$ between these two points.
For another curve $c'$ that intersects $c$ at a single point, 
we may also write, e.g., $c[c',q]$ instead of $c[c \cap c',q]$.
If $c$ is oriented, then we denote by $c^-$ (resp., $c^+$) its starting (resp., terminating) point following its orientation.
We also write $c[-,p]$ and $c[p,+]$ instead of $c[c^-,p]$ and $c[p,c^+]$, respectively.
If $c$ is an $x$-monotone curve and $p$ is a point (not necessarily on $c$),
then with some abuse of notation by writing, e.g., $c[p,+]$, we mean $c[\ell_p \cap c,+\infty)$, where $\ell_p$ is the vertical line through $p$.
As usual, a round bracket indicates that an endpoint does not belong to the subcurve.    

A family of $x$-monotone curves induces a \emph{generalized trapezoidal partition} of the plane in the following way. From every point $p$ which is an endpoint of a curve or an intersection point of two curves we draw a maximal vertical segment (possibly a ray or a line) that contains $p$ and does not cross any other curve but the ones containing $p$.
This yields a partition of the plane into \emph{generalized trapezoids}.

\begin{theorem}[Cutting Lemma for $x$-monotone curves~{\cite[Proposition 2.11]{Sariel00}}]\label{thm:cutting-lemma}
	For every family $\C$ of $n$ $x$-monotone $1$-intersecting curves and any $r>1$ there is a subset of curves in $\C$ which induces a generalized trapezoidal partition of the plane into $O(r^2)$ generalized trapezoids such that the interior of each trapezoid is intersected by at most $n/r$ curves from $\C$.
\end{theorem}

We will also make use of the (asymmetric) K\H{o}v\'ari-S\'os-Tur\'an Theorem
and the recent tight bound of Janzer et al.~\cite{janzer2025} on the size of intersection-reverse sequences.
The latter improves previous bounds by Marcus and Tardos~\cite{MT06} and by Pinchasi and Radoi\v{c}i\'c~\cite{crossingC4}.

\begin{theorem}[K\H{o}v\'ari-S\'os-Tur\'an Theorem~\cite{Kovari}]
	For every $m,n,s,t \ge 1$ if $G=(A \cup B,E)$ is a bipartite graph such that $|A|=m$, $|B|=n$, and there are no $s$ vertices in $A$ and $t$ vertices in $B$ that form a complete bipartite subgraph $K_{s,t}$, then $|E| \le (s-1)^{1/t}nm^{1-1/t}+(t-1)m$.
	In particular, $|E| = O_s(n^{2-1/t})$ when $m=n$ and $s \ge t$.
\end{theorem}

\begin{theorem}[\cite{janzer2025}]
	\label{thm:intersection-reverse}
	Let $A^1,\ldots,A^n$ be linear orders on some subsets of a set of $n$ symbols such that no three symbols appear in the same order in any two distinct linear orders. Then $\sum_{i=1}^n |A^i| = O(n^{3/2})$.
\end{theorem}

\section{Arbitrary $1$-intersecting curves}
\label{sec:1-int}

In this section we consider $1$-intersecting curves that are not necessarily $x$-monotone and prove Theorems~\ref{thm:1-int} and~\ref{thm:precisely-1-int}.
We begin with grounded curves.
A family of curves $\C$ is \emph{grounded} at a curve $c \notin \C$ if every curve in $\C$ has an endpoint on $c$. 

\begin{theorem}\label{thm:2-grounded-disjoint}
	Let $A$ be a family of curves grounded at a curve $\ell_A$ and let $B$ be a family of curves grounded at a curve $\ell_B$, such that all of these curves form a family of $1$-intersecting curves, no curve in $A \cup \{\ell_A\}$ intersects $\ell_B$ and no curve in $B \cup \{\ell_B\}$ intersects $\ell_A$.
	Then the number of tangent pairs $\{c_A \in A, c_B \in B\}$ is $O(|A \cup B|^{3/2})$.  
\end{theorem}

\begin{proof}
	We refer to the curves in $A$ as \emph{red} curves and to the curves in $B$ as \emph{blue} curves.
	If a pair of curves of the same color touch, then we may redraw them locally such that they are disjoint, since we only care about red-blue tangencies.
	Thus, every pair of curves of the same color are either crossing or disjoint.
	Orient each curve in $A$ away from $\ell_A$ and each curve in $B$ away from $\ell_B$.
	Then, there are four types of red-blue tangencies since every curve has a `left' side and a `right' side with respect to its orientation, see Figure~\ref{fig:tangency-types}.
	\begin{figure}[t]
		\centering
		\includegraphics[width=16cm]{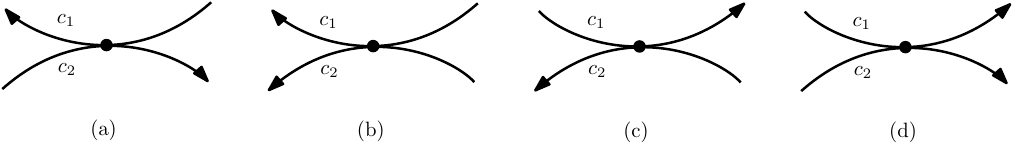}
		\caption{The four tangency types for two oriented curves $c_1$ and $c_2$: (a) left-left (b) left-right (c) right-right (d) right-left.}
	\label{fig:tangency-types}
	\end{figure}
	
	\begin{proposition}\label{prop:no-plus-plus}
		Let $c, c_1, c_2$ and $\ell$ be four distinct curves such that: (1)~$\{c,c_1,c_2,\ell\}$ is $1$-intersecting; (2)~$c_1$ and $c_2$ are grounded at $\ell$ and are oriented away from $\ell$; (3)~$c$ is oriented and touches $c_1$ and $c_2$ at tangency points of the same type; and (4)~$c$ and $\ell$ are disjoint.
		Then $c_1[c,+]$ and $c_2[c,+]$ do not cross.
	\end{proposition}
	
	\begin{proof}
	Suppose that $c_1[c,+]$ and $c_2[c,+]$ cross and let $T$ denote the simple closed curve $c[c_1,c_2] \cup c_1[c,c_2] \cup c_2[c,c_1]$, see Figure~\ref{fig:plus-plus}. 
	Then $c_1^-$ and $c_2^-$ lie on different sides $T$.\footnote{Assuming some arbitrary orientation of $T$, disregarding the orientations of the other curves.}
	However, this is impossible since both of these points are on $\ell$ which cannot cross $T$.
		\begin{figure}[t]
			\centering
			\includegraphics[width= 5cm]{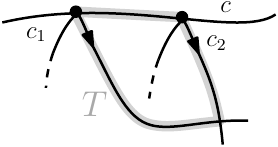}
			\caption{If $c_1$ and $c_2$ are grounded at $\ell$ and have the same tangency type with $c$, then $c_1[c,+]$ and $c_2[c,+]$ cannot cross.}
			\label{fig:plus-plus}
		\end{figure}
	\end{proof}
	
	Returning to the proof of the theorem, we consider each of the four tangency types separately. 
	For a certain type $t$, we write for each curve $b \in B$ the ordered list of curves from $A$ that touch $b$ at a touching point of the given type.
	The list, denote it by $l_t(b)$, is ordered according to the order of the corresponding touching points along $b$.
	
	\begin{proposition}\label{prop:abc-abc}
		There are no three curves $a_1,a_2,a_3 \in A$ and two curves $b_1, b_2 \in B$ such that $a_1,a_2,a_3$ appear in this order both in $l_t(b_1)$ and in $l_t(b_2)$.
	\end{proposition}
	
	\begin{proof}
		Suppose for contradiction that there are such curves.
		Observe first that it follows from Proposition~\ref{prop:no-plus-plus} that $b_1[a_1,+]$ and $b_2[a_1,+]$ do not cross.
		In particular $b_1[a_1,a_3]$ and $b_2[a_1,a_3]$ do not cross.
		Next, we claim that at least one pair of the three subcurves $a_i[b_1,b_2]$, $i=1,2,3$, is crossing.
		Suppose for contradiction that these curves are pairwise disjoint.
		Slightly `inflate' $b_1$ and $b_2$ and fix a point $p(c)$ on each subcurve $c \in \{a_1(b_1,b_2), a_2(b_1,b_2), a_3(b_1,b_2), b_1(a_1,a_3), b_2(a_1,a_3)\}$.
		Note that these subcurves are disjoint.
		Next, draw a crossing-free copy of $K_{2,3}$ whose vertices are these five points and whose edges follow the corresponding subcurves, as suggested by Figure~\ref{fig:crossingC4}.
		\begin{figure}[b]
			\centering
			\includegraphics[width= 8cm]{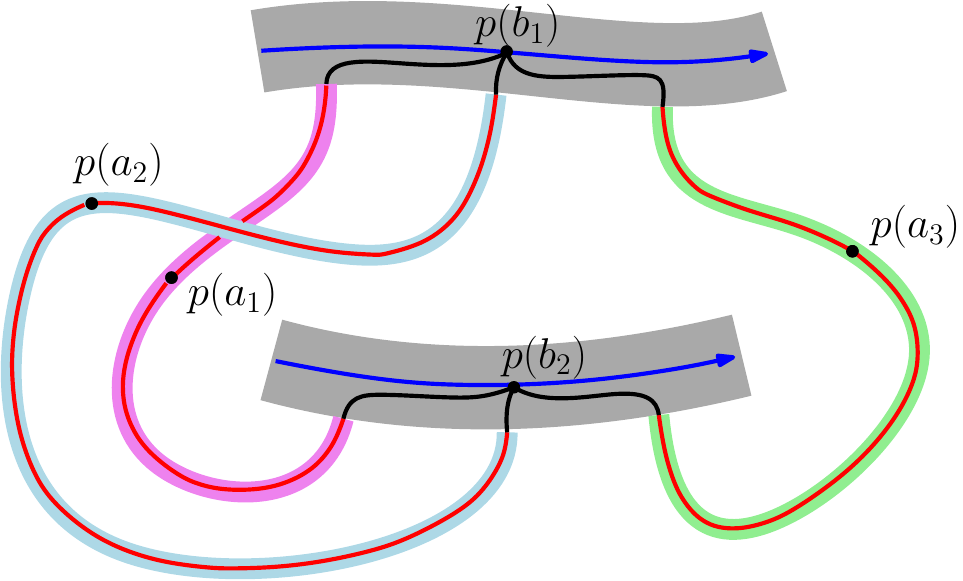}
			\caption{If the subcurves $a_i(b_1,b_2)$ were disjoint, then we could draw a crossing-free copy of $K_{2,3}$.}
			\label{fig:crossingC4}
		\end{figure}	
		Note that since $a_1,a_2,a_3$ touch in this order both of $b_1$ and $b_2$, the counterclockwise cyclic order of $p(a_1),p(a_2),p(a_3)$ is the same for $p(b_1)$ and $p(b_2)$.
		However, this contradicts the following lemma of Pinchasi and Radoi\v{c}i\'c~\cite{crossingC4}.
		
		\begin{lemma}[{\cite[Lemma 1]{crossingC4}}]
			Let $G$ be a topological graph with no two edges belonging to a $4$-cycle that crosses itself an odd number of times.
			For each vertex $v$ of $G$ let $C_v$ denote the counterclockwise cyclic order of the neighbors of $v$.
			Then for each pair of vertices $u$ and $v$, if there are more than two common neighbors of $u$ and $v$, then the order of these neighbors in $C_v$ is reversed with respect to their order in $C_u$. 
		\end{lemma}
		
		For the rest of the proof we will only consider the guaranteed pair of crossing subcurves among $a_i[b_1,b_2]$, $i=1,2,3$, therefore we may
		suppose without loss of generality that $a_1(b_1,b_2)$ and $a_2(b_1,b_2)$ are crossing. 
		We may also assume that $a_1$ meets $b_1$ before meeting $b_2$.
		Since $a_1[b_1,+]$ cannot cross $a_2[b_1,+]$ by Proposition~\ref{prop:no-plus-plus}, it follows that $a_1[b_1,+]$ and $a_2[-,b_1]$ are crossing.
		Together with $b_1[a_1,a_2]$ and $b_2[a_1,a_2]$ they induce a partition of the plane into three connected regions.
		For each type of red-blue tangency, the points $a_1^-$ and $a_2^-$ lie on two different regions, see Figure~\ref{fig:a1a2b1b2} for an illustration.
		\begin{figure}
			\centering
			\includegraphics[width= 12cm]{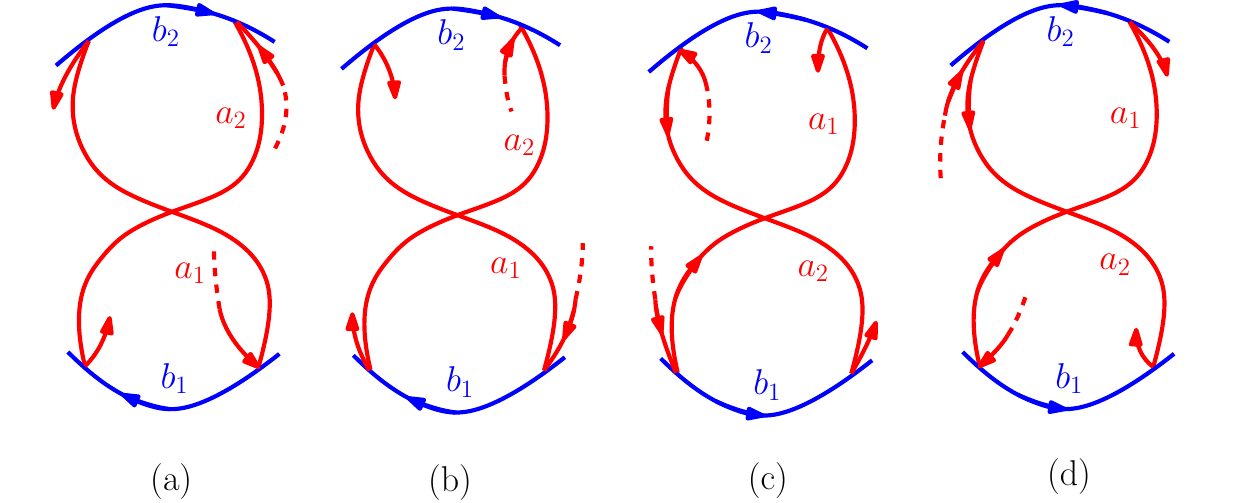}
			\caption{Illustrations for the proof of Proposition~\ref{prop:abc-abc}: $a_1[b_1,+]$ and $a_2[-,b_1]$ are crossing and together with $b_1[a_1,a_2]$ and $b_2[a_1,a_2]$ induce a partition of the plane. $a_1^-$ and $a_2^-$ lie on different regions of this partition for any red-blue tangency type.}
			\label{fig:a1a2b1b2}
		\end{figure}
		However, this is impossible since both of these points lie on $\ell_A$ which does not cross any of the curves $a_1,a_2,b_1,b_2$.
	\end{proof}
	
	It follows from Proposition~\ref{prop:abc-abc} and Theorem~\ref{thm:intersection-reverse} that $\sum_i |l_t(b_i)| = O(|A \cup B|)^{3/2}$. 
	Therefore, the number of red-blue tangencies is also $O(|A \cup B|)^{3/2}$. 
\end{proof}

\begin{corollary}\label{cor:1-grounded}
	Let $\C$ be a family of $n$ $1$-intersecting curves grounded at a curve $\ell$. Then the number of tangencies among $\C$ is $O(n^{3/2})$.  
\end{corollary}

\begin{proof}
	Denote by $t(n)$ the maximum number of tangencies among such a family of $n$ curves. 
	We prove that $t(n)=O(n^{3/2})$ by induction on $n$.
	Split $\ell$ into two disjoint subcurves such that $n/2$ curves from $\C$ are grounded at each subcurve.
	Using Theorem~\ref{thm:2-grounded-disjoint} and induction we get the recursive relation $t(n) \le 2t(n/2) + O(n^{3/2})$ which implies that $t(n) = O(n^{3/2})$.
\end{proof}

\begin{corollary}\label{cor:2-grounded}
	Let $A$ be a family of curves grounded at a curve $\ell_A$ and let $B$ be a family of curves grounded at a curve $\ell_B$, such that $A \cup B \cup \{\ell_A,\ell_B\}$ form a family of $1$-intersecting curves.
	Then the number of tangent pairs $\{c_A \in A, c_B \in B\}$ is $O(|A \cup B|^{3/2})$.  
\end{corollary}

\begin{proof}
	If $\ell_A$ and $\ell_B$ cross at a point $p$, then we split each of them into two curves at $p$ (such that $p$ belongs to neither of them).
	Denote these subcurves by $\ell_A^1, \ell_A^2, \ell_B^1$ and $\ell_B^2$ and note that each pair of them are disjoint.
	Every curve in $A$ (resp., $B$) that crosses $\ell_B$ (resp., $\ell_A$) we cut into two subcurves, such that one of them is grounded at $\ell_A^1$ or $\ell_A^2$ (or at $\ell_A$ if it was not cut) and the other is grounded at $\ell_B^1$ or $\ell_B^2$ (or at $\ell_B$ if it was not cut).
	It remains to bound the number of tangencies among curves grounded at $\ell_x^i$, for each $x \in \{A,B\}$ and $i \in \{1,2\}$, and the number of tangencies of pairs of curves $c$ and $c'$ such that $c$ is grounded at $\ell_x^i$ and $c'$ is grounded at $\ell_y^j$ for each $x,y \in \{A,B\}$ and $i,j \in \{1,2\}$ such that $x \ne y$ or $i \ne j$.
	It follows from Corollary~\ref{cor:1-grounded} and Theorem~\ref{thm:2-grounded-disjoint} that the total number of all of these tangencies is $O(|A \cup  B|^{3/2})$.    
\end{proof}


We can now deduce the main results of this section.
\PWOneIntThm*

\begin{proof}
	Follows immediately from Corollary~\ref{cor:1-grounded}.
\end{proof}

\OneIntThm*

\begin{proof}
Let $\C$ be a family of $n$ $1$-intersecting curves and
let $G_\C$ be the tangency graph of $\C$.
It follows from Corollary~\ref{cor:2-grounded} that $G_\C$ satisfies the $f$-sparse sub-bineighborhoods property for $f(x)=O(x^{3/2})$.
Indeed, let $v_0$ and $v_1$ be two adjacent vertices in $G_\C$ and let $V_i \subseteq N_{G_\C}(v_i) \setminus \{v_{1-i}\}$, for $i=0,1$, be two disjoint subsets of their neighborhoods in $G_\C$.
We wish to bound $t(V_0,V_1)$ --- the number of tangent pairs $\{u_0 \in V_0, u_1 \in V_1\}$.
For every $i=0,1$ and every curve $u \in V_i$, we split $u$ into two subcurves $u_i'$ and $u''_i$ such that each of them is grounded at $v_i$.
Denote the resulting set of curves by $V'_1, V''_1, V'_2$ and $V''_2$, respectively.
Thus, $t(V_0,V_1)=t(V'_0,V'_1)+t(V'_0,V''_1)+t(V''_0,V'_1)+t(V''_0,V''_1)$.
It follows from Corollary \ref{cor:2-grounded} that this sum is upper-bounded by  $O(|V_0 \cup V_1)|^{3/2})$.
Therefore, by Theorem~\ref{thm:spare-neighborhoods-ub} we have that $|E(G_\C)| \le O(n^{5/3})$.
\end{proof}

\section{Tangencies among $x$-monotone curves}
\label{sec:x-mon}

Pach and Sharir~\cite{Pach-Sharir} studied the following problem: 
What is the maximum number of pairs of disjoint line-segments that are \emph{vertically visible} among a set of $n$ segments in the plane?
We say that two segments are \emph{vertically visible} if there exists a vertical segment that intersects both of them and does not intersect any other segment in the set.
The motivation for this problem came from analyzing the maximum size of the events queue in the original implementation of the Bentley-Ottmann line sweeping algorithm for enumerating all the intersections among a set of segments~\cite{Bentely-Ottmann}.

Pach and Sharir~\cite{Pach-Sharir} proved that the maximum number of pairs of disjoint vertically visible line-segments within a set of $n$ segments is $O\left(n^{4/3}\left(\log n\right)^{2/3}\right)$ and $\Omega\left(n^{4/3}\right)$.
They also mentioned these bounds hold also for $x$-monotone $1$-intersecting curves.

It is easy to see that counting disjoint vertically visible pairs can be reduced to counting touching pairs of $x$-monotone $1$-intersecting curves and vice versa.
Indeed, vertically visible curves can become tangent by replacing a very narrow part of one curve by a `spike' that touches the other curve. On the other hand, a pair of touching curves can be easily `detached' and become vertically visible.
Therefore, the result of Pach and Sharir~\cite{Pach-Sharir} implies that the maximum number of tangencies among $n$ $x$-monotone $1$-intersecting curves is $O\left(n^{4/3}\left(\log n\right)^{2/3}\right)$ and $\Omega\left(n^{4/3}\right)$.

\paragraph{Improving Pach and Sharir's upper bound.}
We outline the proof of the upper bound of Pach and Sharir~\cite{Pach-Sharir} and point out the modifications which lead to the better bounds of Theorem~\ref{thm:x-mon} and~\ref{thm:grounded-x-mon}.
First, the cutting lemma (Theorem~\ref{thm:cutting-lemma}) is used to partition the plane into generalized trapezoids such that every trapezoid is cut by not too many curves.
Then, the number of tangencies within every trapezoid $T_i$ is  estimated and finally the sum over all the trapezoids is bounded.
Regarding tangencies within $T_i$, one considers two subsets of curves that cut $T_i$: `short' curves that have at least one endpoint inside $T_i$ and `long' curves that have no endpoint inside $T_i$.
Therefore, the long curves behave like bi-infinite curves with respect to $T_i$.
Thus, the number of long-long tangencies can be bounded using Theorem~\ref{thm:bi-infinite}.
In order to obtain the tight upper bound of Theorem~\ref{thm:grounded-x-mon} we use for this case the linear bound of Lemma~\ref{lemma:blues-above-reds} below.
The number of short-long tangencies is bounded using K\H{o}v\'ari-S\'os-Tur\'an Theorem and a technical lemma (see Lemma~\ref{lem:long-short} below).
Finally, the number of short-short tangencies is bounded in~\cite{Pach-Sharir} by further subdividing $T_i$ into vertical slabs. However, we observe that one can instead use induction and this simplifies the proof and leads to a better upper bound.
For completeness, we include the detailed proofs.


\subsection{Bi-infinite $1$-intersecting $x$-monotone curves}
\label{sec:infinite-mon}
Let $c_1$ and $c_2$ be two bi-infinite $x$-monotone curves and let $v$ be a vertical line.
We say that $v$ intersects $c_1$ \emph{before} it intersects $c_2$ if the intersection point of $c_1$ and $v$ is below the intersection point of $c_2$ and $v$.
$c_1$ \emph{lies below} $c_2$ if there is no vertical line that intersects $c_2$ before it intersects $c_1$.
We say that $c_1$ \emph{starts below} $c_2$ if $c_1$ lies below $c_2$, or $c_1 \cap c_2 \ne \emptyset$
and any vertical line to the left of the leftmost intersection point of $c_1$ and $c_2$ intersects $c_1$ before it intersects $c_2$.

For convenience, we split Theorem~\ref{thm:bi-infinite} into two statements.

\begin{theorem}\label{thm:pw-bi-infinite}
Let $\L$ be a set of $n$ bi-infinite $x$-monotone curves such that every pair of curves in $\L$ intersect at precisely one point.
Then there are at most $n-1$ tangencies among the curves in $\L$.
This bound is best possible.
\end{theorem}

\begin{proof}
It is easy to construct $n-1$ curves such that each of them touches a single additional curve and these $n$ curves are bi-infinite, $x$-monotone and every two of them intersect at exactly one point.
For example, consider the curves that correspond to the graphs of the functions $f(x)=0$ and $f(x)=|x-i|$ for $i=0,1,\ldots,n-2$.

As for the upper bound, let $\L$ be a set of curves as above and observe first that there is no curve that touches one curve from below and another curve from above, for otherwise these two curves do not intersect.
Therefore, $G_\L$ is a bipartite graph and we may refer to the curves that touch some curve from above (resp., below) as the \emph{blue} (resp., \emph{red}) curves. Thus, at every tangency point a red curve touches a blue curve from below.

We show that the tangency graph $G_\L$ is a forest and therefore there are at most $n-1$ tangent pairs.
Suppose for contradiction that there exists a cycle $C$ in $G_\L$.
We may assume without loss of generality that $C$ is a simple cycle, hence, every curve $c$ in $C$ contains two tangency points that correspond to edges in $C$.
We denote by $c_L$ (resp., $c_R$) the left (resp., right) tangency point among them.
Let $b$ be the blue curve in $C$ that starts below every other blue curve in $C$ and let $r$ be the red curve that touches $b$ from below at $b_R$.
	If $b_R=r_R$, then there is a blue curve $b'$ that touches $r$ to the left of $r_{R}$.
However, then $b'$ must intersect $b$ to the right of $b' \cap r$, which implies that $b'$ starts below $b$ (see Figure~\ref{fig:cycle1}).
\begin{figure}[t]
	\centering
	\subfloat[If $b_R=r_R$, then there is a blue curve $b'$ in $C$ that starts below $b$.]{\includegraphics[width= 6cm]{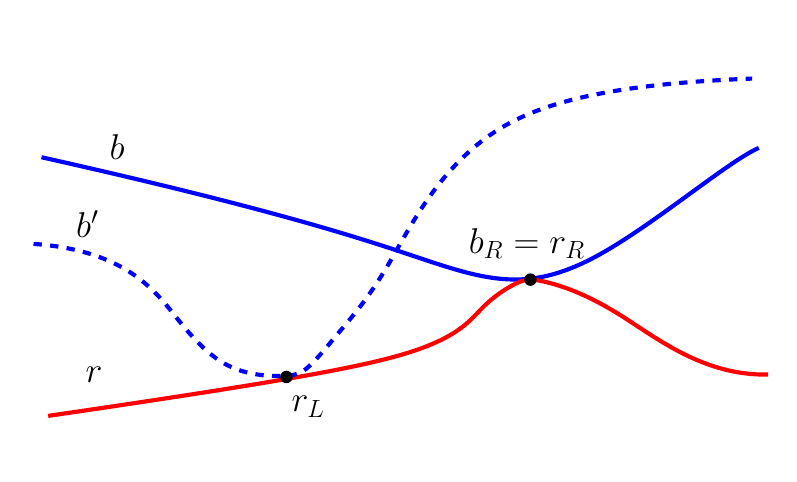}\label{fig:cycle1}}
	\hspace{10mm}
	\subfloat[If $b_R=r_L$, then there are curves $b',r' \in C$ that do not intersect.]{\includegraphics[width= 6cm]{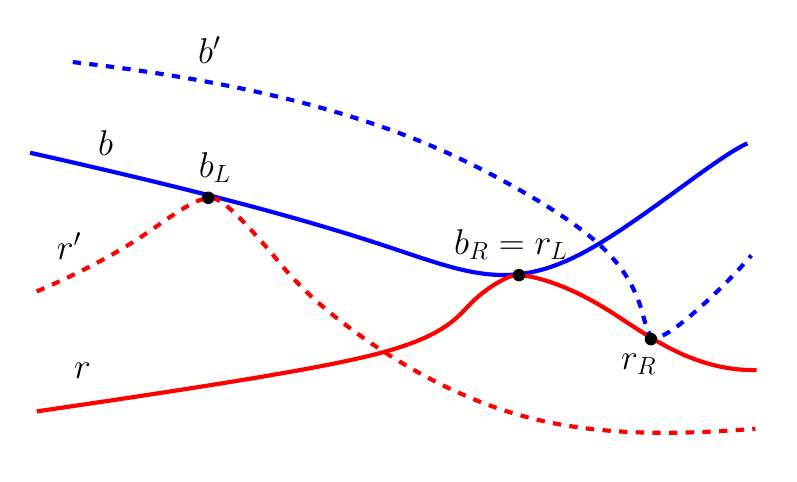}\label{fig:cycle2}}
	\caption{An illustration for the proof of Theorem~\ref{thm:pw-bi-infinite}: If $G_\L$ contains a cycle, then consider the blue curve $b$ that starts below all the other blue curves in the cycle.}
	\label{fig:cycle}		
\end{figure}
If on the other hand $b_R=r_L$, then there is a blue curve $b'$ that touches $r$ to the right of $r_{L}$ and a red curve $r'$ that touches $b$ to the left of $b_R$. However, then $b'$ and $r'$ cannot intersect (see Figure~\ref{fig:cycle2}).
\end{proof}

From Theorem~\ref{thm:pw-bi-infinite} we may conclude a bipartite variant in the spirit of~\cite[Theorem 3.1]{lenses}.

\begin{corollary}
Let $\L=\L_1 \cup \L_2$ be a family of $n$ $x$-monotone bi-infinite $1$-intersecting curves.
Suppose that every curve in $\L_1$ intersects every curve in $\L_2$.
Then the number of tangent pairs $\{\ell_1 \in \L_1,\ell_2 \in \L_2\}$  is at most $2(n-1)$.
\end{corollary}

\begin{proof}
Consider first the number of tangent pair $\{\ell_1 \in \L_1,\ell_2 \in \L_2\}$ such that $\ell_1$ touches $\ell_2$ from below.
We claim that the curves in $\L$ that participate in such touchings are pairwise intersecting and therefore by Theorem~\ref{thm:pw-bi-infinite} there are at most $n-1$ such touchings.
Indeed, suppose that there are two curves $\ell_1,\ell'_1 \in \L_1$ that are disjoint and assume without loss of generality that $\ell_1$ lies above $\ell'_1$.
Let $\ell_2 \in \L_2$ be a curve that touches $\ell_1$ from above.
Then $\ell_2$ and $\ell'_1$ do not intersect, which is a contradiction.
In a similar way it follows that every two curves in $\L_2$ intersect.
Therefore, there are $n-1$ tangent pairs $\{\ell_1 \in \L_1,\ell_2 \in \L_2\}$ such that $\ell_1$ touches $\ell_2$ from below, and, symmetrically, there are $n-1$ tangent pairs $\{\ell_1 \in \L_1,\ell_2 \in \L_2\}$ such that $\ell_1$ touches $\ell_2$ from above.
\end{proof}

For families of bi-infinite $x$-monotone curves such that every two curves intersect at most once we need the following lemma.

\begin{lemma}\label{lemma:blues-above-reds}
Let $\L=\L_1 \cup \L_2$ be a family of $n$ $x$-monotone bi-infinite $1$-intersecting curves.
Suppose that every curve in $\L_1$ starts below every curve in $\L_2$.
Then the number of tangent pairs $\{\ell_1 \in \L_1,\ell_2 \in \L_2\}$  is at most $n-1$.
This bound is best possible.
\end{lemma}

\begin{proof}
Refer to the curves in $\L_1$ as the \emph{red} curves and to the curves in $\L_2$ as the \emph{blue} curves.
We may assume without loss of generality that no two red-blue tangency points have the same $x$-coordinate.
Mark the leftmost red-blue tangency point on every curve.
Since the leftmost red-blue tangency point among these points belongs to two curves, at most $n-1$ points are marked.
We claim that every red-blue tangency point is in fact marked.

Indeed, assume for contradiction that there is an unmarked red-blue tangency point $p$.
Let $r$ and $b$ be the red and blue curves, respectively, that are tangent at $p$. 
Since $p$ is not the leftmost red-blue tangency point on $r$, there is a blue curve $b'$ that touches $r$ to the left of $p$ at a point $q$.
Similarly, there is a red curve $r'$ that touches $b$ to the left of $p$ at a point $z$.
Assume, without loss of generality, that $q$ is to the left of $z$ and refer to Figure~\ref{fig:blues-above-reds}.
\begin{figure}[]
	\centering
	\includegraphics[width= 7cm]{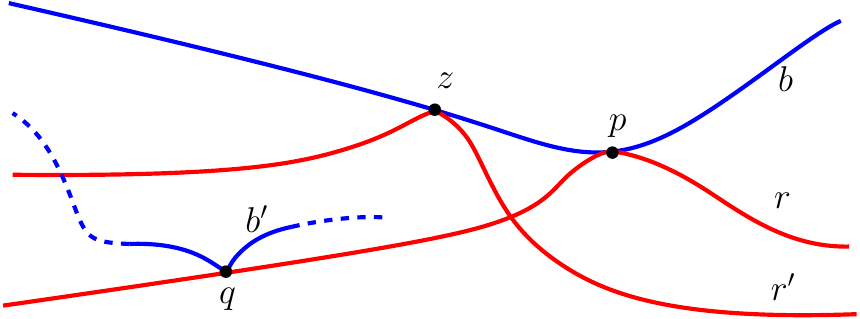}
	\caption{An illustration for the proof of Lemma~\ref{lemma:blues-above-reds}: $p$ is an unmarked red-blue touching point.}
	\label{fig:blues-above-reds}		
\end{figure}
It is easy to see that $r'(z,+)$ and $r(-,p)$ must cross at a point between $z$ and $p$, and therefore $r'$ starts above $r$.
Since $b'$ starts above $r'$, it follows that $b'(-,q)$ crosses $r'(-,q)$.
Therefore, $b'(q,+)$ lies below $r'(q,+)$ which implies that $b'(q,+)$ must cross $r$ since so does $r'(q,+)$.
However, this is impossible since $b'$ and $r$ are tangent.

This concludes the proof of the upper bound.
The construction mentioned in the proof of Theorem~\ref{thm:pw-bi-infinite} shows that this bound is tight.
\end{proof}

Using Lemma~\ref{lemma:blues-above-reds} we obtain a simple proof of the $O(n\log n)$ upper bound of Theorem~\ref{thm:bi-infinite}.
This bound was already proved by Pach and Sharir~\cite[Theorem 2]{Pach-Sharir}, however, using less elementary tools.
For completeness we also describe the lower bound construction in terms of tangent curves.

\begin{theorem}\label{theorem:bi-infinite}
Let $\L$ be a family of $n$ $x$-monotone bi-infinite $1$-intersecting curves.
Then the number of tangencies among the curves in $\L$ is $O(n\log n)$.
This bound is asymptotically tight.
\end{theorem}

\begin{proof}
Let $f(n)$ be the maximum number of tangencies in a family of $n$ $x$-monotone bi-infinite $1$-intersecting curves.
We prove that $f(n) \le n \log_{2} n$ by induction.
The claim clearly holds for $n=1$.
For $n>1$, let $\L=\{\ell_1,\ell_2,\ldots,\ell_n\}$ be a family of $x$-monotone bi-infinite $1$-intersecting curves such that $\ell_i$ starts below $\ell_j$ for every $i<j$.
Set $\L_1=\{\ell_1,\ldots,\ell_{\lfloor n/2 \rfloor}\}$ and $\L_2=\{\ell_{\lfloor n/2 \rfloor +1},\ldots,\ell_{n}\}$.
By the induction hypothesis there are at most $f(\lfloor \frac{n}{2} \rfloor) \le \lfloor \frac{n}{2} \rfloor \log_2 \lfloor \frac{n}{2} \rfloor$ tangencies among the curves in $\L_1$ and at most $f(\lceil \frac{n}{2} \rceil) \le \lceil \frac{n}{2} \rceil \log_2 \lceil \frac{n}{2} \rceil$ tangencies among the curves in $\L_2$.
By Lemma~\ref{lemma:blues-above-reds} there are at most $n-1$ touching pairs of curves such that one of them belongs to $\L_1$ and the other belongs to $\L_2$.
Therefore the number of tangencies among the curves in $\L$ is at most $\lfloor \frac{n}{2} \rfloor \log_2 \lfloor \frac{n}{2} \rfloor + \lceil \frac{n}{2} \rceil \log_2 \lceil \frac{n}{2} \rceil + n-1 \le n\log_2 n$.

As for the lower bound, we show that for any positive integer $k$ we have $f(2^k) \ge 2^{k-1}k$.
For $k=1$ this clearly holds.
Assume that the claim is true for $k$. Then we can combine two families of $2^k$ curves with at least $2^{k-1}k$ tangencies within each family in the way suggested in Figure~\ref{fig:nlogn-lb}. 
The total number of tangencies among these $2^{k+1}$ curves is therefore $2\cdot2^{k-1}k + 2^k = 2^k(k+1)$.
\begin{figure}[t]
	\centering
	\includegraphics[width= 7cm]{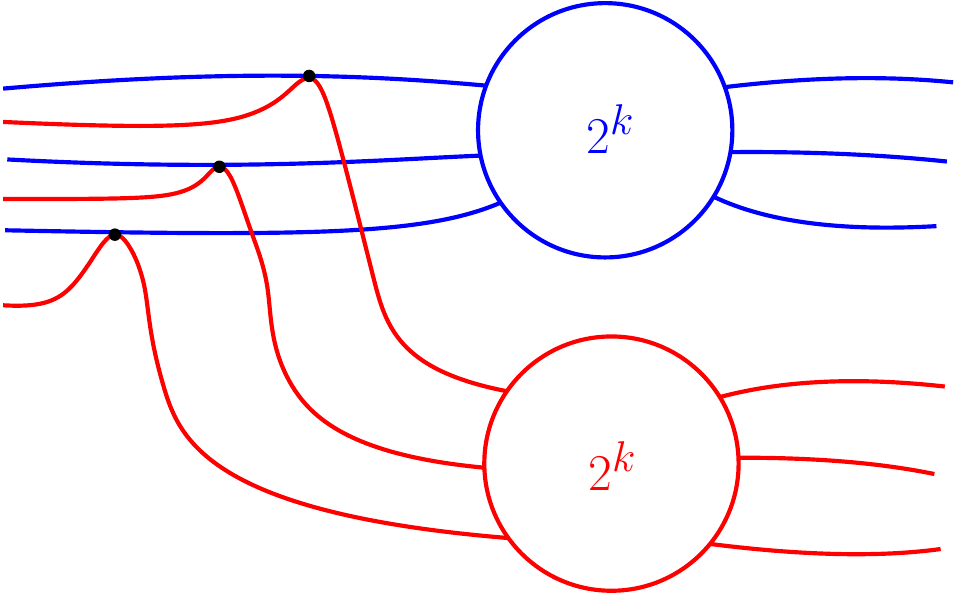}
	\caption{$2^{k+1}$ curves admitting $2^k(k+1)$ tangencies.}
	\label{fig:nlogn-lb}		
\end{figure}
\end{proof}
 
\subsection{Arbitrary $1$-intersecting $x$-monotone curves}
\label{sec:mon}



In this section we prove Theorem~\ref{thm:x-mon}. 
Most ingredients already appear in~\cite{Pach-Sharir}, however, we reproduce them here in order to get a complete proof in the context of curves of our improved upper bound.

\begin{lemma}\label{lem:long-short}
Let $\L=\{a,b,c,d\}$ be a set of four bi-infinite $x$-monotone curves.
Then there are no two $x$-monotone curves $e,f$ such that each of them is tangent to every curve in $\L$ from below and $\L \cup \{e,f\}$ is $1$-intersecting.
\end{lemma}

\begin{proof}
Suppose for contradiction that there exist six curves as above.
Since $e$ and $f$ touch each of $a,b,c,d$ from below, every curve in $\L$ appears on the lower envelope of $\L$.
Assume without loss of generality that $a,b,c,d$ appear in this order from left to right on the lower envelope of $\L$ and let $p=b \cap c$, see Figure~\ref{fig:longs-shorts}.
\begin{figure}[t]
	\centering
	\includegraphics[width= 9cm]{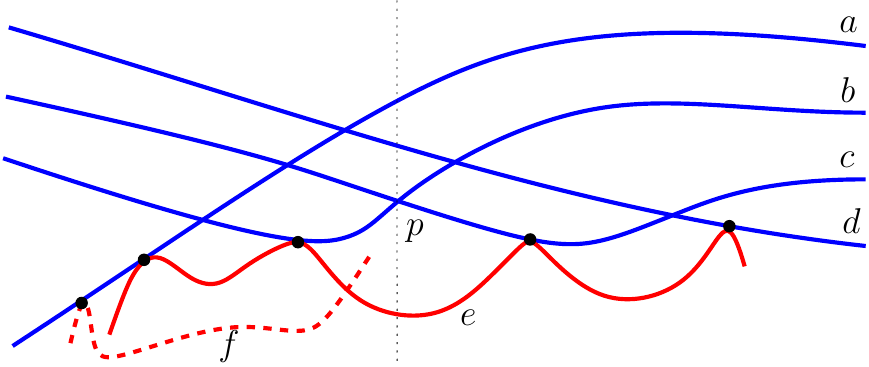}
	\caption{An illustration for the proof of Lemma~\ref{lem:long-short}: if $f \cap a$ is to the left of $e \cap a$, then $f(f \cap a,p)$ must cross $e(e \cap a,p)$ for otherwise $f$ cannot touch $b$.}
	\label{fig:longs-shorts}
\end{figure}

Observe that if $f \cap a$ is to the left of $e \cap a$, then $f(f \cap a,p)$ must cross $e(e \cap a,p)$ for otherwise $f$ cannot touch $b$. Therefore $e(-,p)$ and $f(-,p)$ must intersect.
If $f \cap a$ is to the right of $e \cap a$, then again by symmetry $e(-,p)$ and $f(-,p)$ must intersect.
Similarly, $e(p,+)$ and $f(p,+)$ must intersect, and thus $e$ and $f$ intersect more than once, a contradiction.
\end{proof}

\xmonThm*


\begin{proof}
We prove by induction on $n$ that there at most $c n^{4/3} (\log n)^{1/3}$ tangencies for some absolute constant $c$. 
We choose (later) $n_0$ and $c$ such that the claim holds trivially for $n \le n_0$.
Let $\C$ be a family of $n$ $1$-intersecting $x$-monotone curves where $n > n_0$ and assume that the claim holds for every such family with fewer curves.
Denote by $t(\C)$ the number of tangencies among the curves in $\C$
and set $r=\frac{n^{1/3}}{(\log n)^{2/3}}$.
By Theorem~\ref{thm:cutting-lemma} there is a subset $\C' \subseteq \C$ that induces a partition of the plane into $O\left(\frac{n^{2/3}}{(\log n)^{4/3}}\right)$ generalized trapezoids such that each trapezoid intersects at most $n^{2/3} (\log n)^{2/3}$ curves from $\C$.

Next, we bound the number of tangencies within every trapezoid $t_i$.
Let $\L_i$ be the curves that intersect $t_i$ and both of their endpoints are not in the interior of $t_i$ and let $\S_i$ be the curves that intersect $t_i$ and at least one of their endpoints is in the interior of $t_i$.
Note that $\L_i$ contains at most two curves in $\C'$ that bound $t_i$ from above and below.
Observe also that for every curve $c \in \L_i \cup \S_i$ we have that $c \cap t_i$ consists of a single curve, since the curves are $x$-monotone and each of them intersects at most once the two curves in $\C'$ that bound $t_i$ from below and from above.
Disregard the curves in $\C \setminus (\L_i \cup \S_i)$ and
consider the clipping of the curves in $\L_i$ and $\S_i$ to $t_i$. 
Observe that the curves that correspond to $\L_i$ (clipped to $t_i$) can be easily extended to $x$-monotone bi-infinite curves without introducing new intersection points.
By abuse of notation we denote the new sets of curves by $\L_i$ and $\S_i$.
Let $m_i = |\L_i|$ and $n_i = |\S_i|$.

We bound the number of tangencies within each of the sets $\L_i$ and $\S_i$ and the number of tangencies between curves of different sets:
By Theorem~\ref{theorem:bi-infinite} there are $O(m_i \log m_i)$ tangencies among the curves in $\L_i$.
By induction, there are at most $c n_i^{4/3} (\log n_i)^{1/3}$ tangencies among the curves in $\S_i$. 
By symmetry, it is enough to bound the number of tangent pairs $s_i \in \S_i$, $l_i \in \L_i$ such that $s_i$ touches $l_i$ from below. Consider the bipartite tangency graph $(\L_i \cup \S_i, \{(l_i\in\L_i,s_i\in\S_i) \mid \textrm{$l_i$ and $s_i$ are tangent and $s_i$ touches $l_i$ from below}\})$. 
It follows from Lemma~\ref{lem:long-short} that this graph does not contain four vertices in $\L_i$ such that each of them is adjacent to the same two vertices in $\S_i$. Therefore by the K\"ov\'ari-S\'os-Tur\'an Theorem this graph has at most $O(n_i\sqrt{m_i}+m_i)$ edges.

Thus, $t(\C)$ is at most
\begin{equation}
\sum_{i=1}^{O(r^2)} \left(O(m_i\log m_i) +O(n_i\sqrt{m_i}+m_i) + c n_i^{4/3} (\log n_i)^{1/3}\right).
\end{equation}
Recall that $n_i \le n/r$ and $m_i \le n/r+2 \le 2n/r$.
We also have that $\sum n_i \le 2n$. Therefore, 
\begin{equation}
\sum_{i=1}^{O(r^2)} O(m_i\log m_i) \le O\left(\frac{n^{2/3}}{(\log n)^{4/3}} \frac{2n(\log n)^{2/3}}{n^{1/3}} \log n \right) \le O\left( n^{4/3}(\log n)^{1/3}\right),
\end{equation}
\begin{equation}
	\sum_{i=1}^{O(r^2)} O(n_i\sqrt{m_i}+m_i) \le O\left(2n \sqrt{\frac{2n(\log n)^{2/3}}{n^{1/3}}} + \frac{n^{2/3}}{(\log n)^{4/3}} \frac{2n(\log n)^{2/3}}{n^{1/3}} \right) \le O\left(n^{4/3}(\log n)^{1/3}\right)
\end{equation}
and,
\begin{equation}
	\sum_{i=1}^{O(r^2)} cn_i^{4/3}(\log n_i)^{1/3} \le c\cdot 2n(n/r)^{1/3} (\log n)^{1/3} 
	\le 2c n^{11/9} (\log n)^{5/9}.
\end{equation}

Combining all the above there is a constant $c'$ such that $t(\C) \le c'n^{4/3} (\log n)^{1/3} + 2c n^{11/9} (\log n)^{5/9}$.
If we choose $n_0$ large enough such that $2(\log n)^{2/9} \le n^{1/9}/2$, then we get that
$t(\C) \le c'n^{4/3} (\log n)^{1/3} + \frac{c}{2} n^{4/3} (\log n)^{1/3} \le cn^{4/3} (\log n)^{1/3}$, where the last inequality holds for $c \ge 2c'$.
\end{proof}

\subsection{Grounded $1$-intersecting $x$-monotone curves}
\label{sec:grounded-mon}

In this section we prove Theorem~\ref{thm:grounded-x-mon}.
First we prove a special case using similar arguments to the ones in the proof of Theorem~\ref{thm:x-mon}, then we conclude the general case by induction.

\begin{theorem}\label{thm:red-blue-x-mono-grounded}
	Let $\C = \C_1 \cup \C_2$ be a family of $n$ $1$-intersecting $x$-monotone curves such that the left endpoint of every curve in $\C$ lies on a common vertical line $\ell$ and for every pair of curves $c_1 \in \C_1$ and $c_2 \in C_2$ the left endpoint of $c_1$ is lower than the left endpoint of $c_2$.
	Then, there are at most $d n^{4/3}$ tangent pairs $\{c_1 \in C_1, c_2 \in C_2\}$ for some absolute constant $d$.
\end{theorem}

\begin{proof}
	We prove by induction on $n$. 
	We choose (later) $n_0$ and $d$ such that the claim holds trivially for $n \le n_0$.
	Let $\C$ be a family of $n$ curves as in the statement of the theorem where $n > n_0$ and assume that the claim holds for every such family with fewer curves.
	We refer to the curves in $\C_1$ and $\C_2$ as \emph{red} and \emph{blue}, respectively.
	Denote by $t(\C)$ the number of red-blue tangencies among the curves in $\C$
	and set $r=n^{1/3}$.
	By Theorem~\ref{thm:cutting-lemma} there is a subset $\C' \subseteq \C$ that induces a partition of the plane into $O(n^{2/3})$ generalized trapezoids such that each trapezoid intersects at most $n/r=n^{2/3}$ curves from $\C$.
	
	Next, we bound the number of red-blue tangencies within every trapezoid $t_i$.
	Let $\L_i$ be the curves that intersect $t_i$ and both of their endpoints are not in the interior of $t_i$ and let $\S_i$ be the curves that intersect $t_i$ and one of their endpoints is in the interior of $t_i$.
	Recall that $\L_i$ contains at most two curves in $\C'$ that bound $t_i$ from above and below.
	Recall also that for every curve $c \in \L_i \cup \S_i$ we have that $c \cap t_i$ consists of a single curve, since the curves are $x$-monotone and each of them intersects at most once the two curves in $\C'$ that bound $t_i$ from below and from above.
	Disregard the curves in $\C \setminus (\L_i \cup \S_i)$ and
	consider the part of each curve in $\S_i$ between its leftmost and its rightmost intersection point with $t_i$.
	As for the curves in $\L_i$, we consider their part from $\ell$ to their rightmost intersection with $t_i$ and extend them leftwards and rightwards by horizontal rays into bi-infinite $x$-monotone curves.
	By abuse of notation we denote the new sets of curves by $\L_i$ and $\S_i$.
	Let $m_i = |\L_i|$ and $n_i = |\S_i|$.
	
	We bound the number of red-blue tangencies within each of the sets $\L_i$ and $\S_i$ and the number of tangencies between curves of different sets:
	By Lemma~\ref{lemma:blues-above-reds} there are at most $m_i-1$ red-blue tangencies among the curves in $\L_i$. 
	By induction, there are at most $d n_i^{4/3}$ red-blue tangencies among the curves in $\S_i$. 
	Consider the bipartite tangency graph $(\L_i \cup \S_i, \{(l_i\in\L_i,s_i\in\S_i) \mid \textrm{$l_i$ and $s_i$ form a red-blue tangency such that $s_i$ is red and $l_i$ is blue}\})$.
	By symmetry, it is enough to bound the number of this type of red-blue tangencies.
	It follows from Lemma~\ref{lem:long-short} that this graph does not contain
	four vertices in $\L_i$ such that each of them is adjacent to the same two vertices in $\S_i$. Therefore by the K\H{o}v\'ari-S\'os-Tur\'an Theorem~\cite{Kovari} this graph has at most $O(n_i\sqrt{m_i}+m_i)$ edges.
		
	Thus,
	\begin{equation}
		t(\C) \le \sum_{i=1}^{O(r^2)} \left(m_i +O(n_i\sqrt{m_i}+m_i) + d n_i^{4/3} \right).
	\end{equation}
	Recall that $n_i \le n/r$ and $m_i \le n/r+2 \le 2n/r$.
	We also have that $\sum n_i \le n$. Therefore, 
	\begin{equation}
		\sum_{i=1}^{O(r^2)} O(n_i\sqrt{m_i}+m_i) \le O(n\sqrt{2n^{2/3}} + n^{2/3}\cdot2n^{2/3}) \le O(n^{4/3})
	\end{equation}
	and,
	\begin{equation}
		\sum_{i=1}^{O(r^2)} dn_i^{4/3} \le d(n/r)^{1/3} \sum_{i=1}^{O(r^2)} n_i \le dn^{11/9}.
	\end{equation}
	
	Combining the above, there is a constant $d'$ such that $t(\C) \le d'n^{4/3} + d n^{11/9} $.
	If we choose $n_0$ large enough such that $n^{1/9}/2 \ge 1$, then we get that
	$t(\C) \le d'n^{4/3} + \frac{d}{2} n^{4/3} \le dn^{4/3}$, where the last inequality holds for $d \ge \max\{2d', n_0^{2/3}/2\}$.
\end{proof}

	From Theorem~\ref{thm:red-blue-x-mono-grounded} we can conclude the same upper bound for grounded $x$-monotone curves.

\groundedXmonThm*

\begin{proof}
	Let $t(n)$ be the maximum number of tangencies among a family of $n$ curves as in the statement of the theorem.
	By induction and Theorem~\ref{thm:red-blue-x-mono-grounded} we get the recursive relation $t(n) \le 2t(n/2) + O(n^{4/3})$ whose solution is $t(n) \le O(n^{4/3})$.
	
	As for the lower bound, consider the famous construction of $n$ lines and $n$ points that exhibits $\Omega(n^{4/3})$ point-line incidences (see, e.g., \cite{pachbook}).
	Let $\ell$ be a vertical line to the left of all the points.
	At each point $p$ of the set of points fix the center of a small disk $D_p$ and redraw the lines incident to $p$, denote them by $\L_p$, as $1$-intersecting $x$-monotone curves starting at $\ell$ such that each of them contributes a segment to their upper envelope which lies within $D_p$. 
	Replace $p$ with an $x$-monotone curve $c_p$ that follows the line with the smallest slope in $\L_p$ from $\ell$ to $D_p$ at some small distance, and then touches each of the curves in $\L_p$. Curves that follow the same line are drawn at different distances from that line such that they do intersect. See Figure~\ref{fig:gounded-mono-lb} for an illustration.
	\begin{figure}[t]
		\centering
		\includegraphics[width= 7cm]{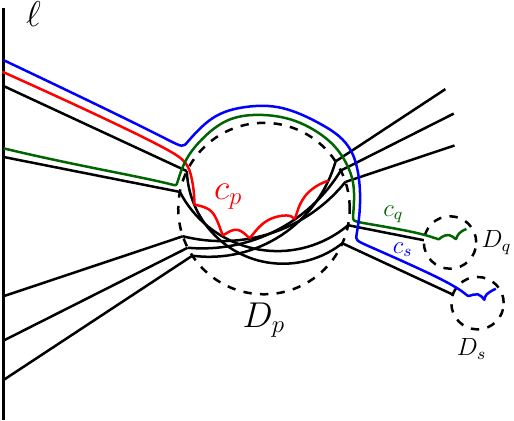}
		\caption{A construction of $n$ grounded bi-infinite $x$-monotone $1$-intersecting curves admitting $\Omega(n^{4/3})$ tangencies.}
		\label{fig:gounded-mono-lb}		
	\end{figure}
\end{proof}

\begin{corollary}
	The maximum number of tangencies among a family of $n$ $1$-intersecting $x$-monotone curves such that the intersection of their projections on the $x$-axis is nonempty is $\Theta(n^{4/3})$.
\end{corollary}

\subsection{Bi-infinite $k$-intersecting $x$-monotone curves}
\label{sec-bi-infinite-k}

In light of Theorem~\ref{thm:pw-bi-infinite} and ~\cite[Theorem 2.4]{lenses} it is tempting to suggest that a similar statement holds for every fixed $k$.

\begin{conjecture}\label{conj:k-intersecting-x-monotone}
Every set of $n$ pairwise intersecting bi-infinite $x$-monotone $k$-intersecting curves admits $O_k(n)$ tangencies.
\end{conjecture}

We were unable to settle this conjecture.
Still, we mention a simple and weak polynomial upper bound and two observations regarding the conjecture.
For the upper bound we use the following observation.

\begin{observation}\label{obs:bi-infinite}
	Let $r$ and $r'$ be two bi-infinite $x$-monotone curves touching from below an $x$-monotone curve $b$. Then $r$ and $r'$ cross at a point whose $x$-coordinate is between $b \cap r$ and $b \cap r'$.
\end{observation}

By Theorem~\ref{thm:KP} we have $O_k\left(n^{2-\frac{1}{k+3}}\right)$ tangencies among any $n$ $k$-intersecting curves.
One can slightly improve upon this bound in case of bi-infinite $x$-monotone curves.

\begin{proposition}\label{prop:bi-infinite-k-intersecting-bound}
	Let $\L$ be a family of $k$-intersecting bi-infinite $x$-monotone curves.
	Then $\L$ admits $O_k\left(n^{2-\frac{1}{k+1}}\right)$ tangencies.
\end{proposition}

\begin{proof}
	Note that the curves in $\L$ are not necessarily pairwise intersecting.
Consider the tangency graph $G_\L$ and flip a fair coin for every curve $c \in \L$. If the outcome is `heads', delete all the edges that correspond to tangencies in which $c$ touches another curve from below.
Otherwise, if the outcome is `tails' do the same for tangencies in which $c$ touches a curve from above.
Clearly, every edge in $G_\L$ survives with probability $1/4$. Therefore, there is a bipartite  subgraph $(B \cup R, E)$ with at least $|E(G_\L)|/4$ edges in which every curve that corresponds to a vertex $r \in R$ touches from below the curves corresponding to its neighbors in $B$.
It suffices to bound the size of this subgraph.
For convenience, we perturb the curves in $\L$, detaching tangencies that do not correspond to edges of $(B \cup R, E)$. 

	We claim that $(B \cup R, E)$ does not contain $K_{k+1,c_k}$ as a subgraph, for some large enough $c_k$ to be determined later.
	Indeed, suppose without loss of generality that there is a subset of $c_k$ curves $R'=\{r_1,r_2,\ldots,r_{c_k}\}\subseteq R$ each of which touches from below each of the $k+1$ curves in a subset $B'=\{b_1,b_2,\ldots,b_{k+1}\}\subseteq B$.
	Let $\ell$ be the lower envelope of the curves in $B'$.
	Then $\ell$ consists of at most $\lambda_{k+2}(k+1)$ curve-segments, where $\lambda_s(n)$ denotes that maximum length of an $(n,s)$-Davenport-Schinzel sequence~\cite{DS}.\footnote{A sequence of letters over an $n$-element alphabet is an $(n,s)$-\emph{Davenport-Schinzel sequence} if it does not contain two consecutive identical letters and no alternating sub-sequence of two letters whose length is $s+2$. In fact a trivial bound of $O(k^3)$ on the size of the lower envelope also suffices for our purpose.}
	Note that each curve $r' \in R'$ touches $k+1$ of these curve-segments (which belong to $k+1$ different curves in $B'$).
	Therefore, there are at most $\binom{\lambda_{k+2}(k+1)}{k+1}$ possibilities for the $(k+1)$-subset of curve-segments of $\ell$ that a curve in $R'$ touches.
	Setting $c_k = 1+ \binom{\lambda_{k+2}(k+1)}{k+1}$
	implies that there are two curves $r',r'' \in R'$ that touch the same $k+1$ curve-segments of $\ell$.
	Consider one of these curve-segments $s$. Then it follows from Observation~\ref{obs:bi-infinite} that $r'$ and $r''$ cross at a point whose $x$-coordinate is between $s \cap r'$ and $s \cap r''$.
	Since the $k+1$ curve-segments that $r'$ and $r''$ touch have distinct $x$-projections, it follows that $r'$ and $r''$ cross at least $k+1$ times, which is impossible.
	Thus, $(B \cup R,E)$ is a $K_{k+1,c_k}$-free graph and it follows from K\H{o}v\'ari-S\'os-Tur\'an Theorem that it has at most $O_k\left(n^{2-\frac{1}{k+1}}\right)$ edges.
\end{proof}

The next observation says that if Conjecture~\ref{conj:k-intersecting-x-monotone} is true, then the same statement holds without requiring the curves to be bi-infinite.

\begin{proposition}
Suppose that Conjecture~\ref{conj:k-intersecting-x-monotone} holds for some $k>2$.
Then every set of $n$ pairwise intersecting $x$-monotone $(k-2)$-intersecting curves admits $O_k(n)$ tangencies.
\end{proposition}

\begin{proof}
Let $\L$ be a set of curves as in the proposition.
We assume without loss of generality that no two endpoints of two curves share the same $x$-coordinate.
There are four types of tangencies: the curve touching from above starts before/after the curve touching from below and ends before/after it.
We consider each type separately and observe that the set of curves touching from above and the set of curves touching from below are disjoint sets.
Call the curves in the first set \emph{red} and the curves in the other set \emph{blue}.
Next, we extend each red curve into a bi-infinite curve by shooting a very steep ray leftwards (resp., rightwards) and upwards at its left (resp., right) endpoint.
Similarly, we extend each blue curve into a bi-infinite curve by shooting a very steep ray leftwards (resp., rightwards) and downwards at its left (resp., right) endpoint.
The rays going leftward and upwards and the rays going rightwards and downwards have opposite slopes.
Similarly, the rays going leftward and downwards and the rays going rightwards and upwards have opposite slopes.
Note that this introduces at most two new intersection points between two curves while for touching curves (of the considered type) no new intersection point is introduced.
\end{proof}

According to the following observation for an even $k$ we may assume that every non-tangent pair of curve intersects at exactly $k$ points.

\begin{proposition}
Let $\L$ be a set of $n$ bi-infinite $x$-monotone curves such that every pair of curves in $\L$ intersect at at least one and at most $k$ points for some even $k$.
Then there is a set $\L'$ of $n$ bi-infinite $x$-monotone curves such that every pair of curves in $\L'$ is either touching at a single point or crossing at exactly $k$ points and the number of touching pairs in $\L$ and $\L'$ is the same.
\end{proposition}

\begin{proof}
We may assume without loss of generality that all the curves are parallel to the $x$-axis at $-\infty$ and $+\infty$ and denote them by $1,2,\ldots,n$ according to their order from top to bottom at $-\infty$.
Let $f:\{1,2,\ldots,n\}\rightarrow \{1,2,\ldots,n\}$ denote their order at $+\infty$, that is, $f(i)$ is the location of $i$ in the order of the curves from top to bottom at $+\infty$.

If $f$ is the identity function, then every pair of curves either touch or cross at an even number of points.
In such a case, every two curves that cross at $k'<k$ points, can be redrawn very close to one of their crossing points such that they cross at $k-k'$ new crossing points.
This way we obtain exactly $k$ crossing points between every pair of non-touching curves.

Otherwise, if $f$ is not the identity function, then we can find two curves that cross an odd number of times, redraw them such that they cross an even number of times and repeat this process until $f$ becomes the identity function.
Indeed, let $i$ and $j$ be two curves such that $i<j$ and $f(i) > f(j)$.
If there are several such pairs we choose a `closest' one, that is, one which minimizes the difference $f(i)-f(j)$.
Note that there is no curve $l$ such that $f(i) > f(l) > f(j)$, for otherwise we must have both $l<i$ and $j<l$ (or else $(i,l)$ or $(l,j)$ would be a `reversed' pair closer than $(i,j)$) which is impossible since $i<j$.
Therefore, $f(i)=f(j)+1$ and it follows that we can redraw $i$ and $j$ such that they cross at an additional point. 
\end{proof}


\section{An Erd\H{o}s-Simonovits-type theorem}
\label{sec:sparse-neighborhoods}

In this section we prove Theorems~\ref{thm:spare-neighborhoods-ub} and~\ref{thm:spare-neighborhoods-lb} and discuss some of their applications and connections to other graph theoretic results.
For a graph $G=(V,E)$ we denote by $\overline{d}(G)=2|E|/|V|$ the average degree in $G$. 
The subgraph induced by a subset of vertices $U \subseteq V$ is denoted by $G[U]$.
For a vertex $v \in V$ we denote by $N_G(v)$ (or simply $N(v)$ when the context is clear) the neighbors of $v$ in $G$.
The (\emph{open}) \emph{bineighborhood} of two vertices $v$ and $u$ is equal to $(N_G(v) \cup N_G(u)) \setminus \{u,v\}$.
Assume henceforth that $G$ is bipartite and 
recall that for a nonnegative and nondecreasing function $f:\mathbb{R} \rightarrow \mathbb{R}$ we say that two vertices $u$ and $v$ have \emph{$f$-sparse sub-bineighborhoods}
if for every two disjoint subsets $U \subseteq N_G(u) \setminus \{v\}$ and $V \subseteq N_G(v) \setminus \{u\}$ it holds that $|E(G(U,V))| \le f(|U \cup V|)$, where $G(U,V)$ is the bipartite subgraph $(U \cup V, \{(u,v) \in E(G) \mid u\in U \textrm{ and } v\in V\})$.
Theorem~\ref{thm:spare-neighborhoods-ub} will follow from the next result.

\begin{theorem}
	\label{thm:f-sparse-neighborhoods}
	Let $f:\mathbb{R} \rightarrow \mathbb{R}$ be a nonnegative and nondecreasing function and let $G=(A \cup B,E)$ be an $n$-vertex bipartite graph with $\overline{d}(G) \ge 16$. \\
	(i)~If every pair of vertices in $G$ has $f$-sparse sub-bineighborhoods, then $(\overline{d}(G))^3 \le 2^{11}nf(2\overline{d}(G))$.\\
	(ii)~If every pair of adjacent vertices has $f$-sparse sub-bineighborhoods and $|E| \ge 25n^{3/2}$, then $(\overline{d}(G))^3 \le 2^{20}nf(2\overline{d}(G))$.
\end{theorem}

\begin{proof}
	Set $m=|E|$ and let $d=\overline{d}(G)=2m/n$ be the average degree of $G$.
	We first replace $G$ with a graph $G'$ in which the degree of every vertex is roughly the average degree (while maintaining the $f$-sparseness property):
	Go over the vertices of $G$ in an arbitrary order and replace each vertex $v$ such that $\deg(v)>d$ by $\lfloor \deg(v)/d \rfloor$ vertices of degree $d$ and possibly one vertex of a smaller degree.
	The neighbors of $v$ are connected to the vertices that replace $v$ in an arbitrary manner which satisfies this property.
	Observe that the resulting graph still has $f$-sparse sub-bineighborhoods for every pair of (adjacent) vertices, since each sub-neighborhood of a vertex in the new graph cannot contain two vertices that correspond to the same original vertex.
	Thus, every sub-neighborhood in the new graph corresponds to a sub-neighborhood in $G$.
	Furthermore, since $d=2m/n$, the number of vertices in the new graph is at most $\sum_{v \in A\cup B} (\lfloor \deg_{G}(v)/d \rfloor + 1) \le 2n$. 
	Because the number of edges remains the same, the average degree of this graph is at least $d/2$.
	
	Next, we repeatedly remove vertices of degree smaller than $d/8$ until no such vertices remain.
	We claim that at least $m/2$ edges remain after this step.
	Indeed, suppose for contradiction that less than $m/2$ edges remain and denote by $x$ the number of vertices that were removed.
	Then at most $xd/8$ edges were removed, therefore $xd/8>m/2=2nd/8$ which is impossible since $x \le 2n$.	
	Denote by $G'=(A' \cup B', E')$ the resulting graph and note that $|A' \cup B'| \le 2n$, $|E'| \ge m/2$ and the degree of every vertex in $G'$ is at least $d/8$ and at most $d$.
	Observe also that  $|A' \cup B'| \ge 2|E'|/d \ge m/d = n/2$.
	Furthermore, $G'$ still has the $f$-sparse sub-bineighborhoods property, since this property is maintained when deleting a vertex.
	
	Considering (i), for each pair of vertices $a \in A'$ and $b \in B'$ the number of edges in the subgraph of $G'$ induced by the union of their open bineighborhoods is at most $f(2d)$.
	Therefore we have $\sum_{a \in A', b \in B'}  |E(G'[(N_{G'}(a)\setminus \{b\}) \cup (N_{G'}(b) \setminus \{a\}])| \le \binom{2n}{2} f(2d) \le 2n^2 f(2d)$.
	Note that every edge $(a,b) \in E'$ is counted in this sum for every pair $a' \in A'$ and $b' \in B'$ such that $a \in N_{G'}(b') \setminus \{a'\}$ and $b \in N_{G'}(a') \setminus \{b'\}$. Since $a$ (resp., $b$) has at least $\frac{d}{8}-1 \ge \frac{d}{16}$ neighbors different from $b$ (resp., $a$), it follows that $(a,b)$ is counted at least $d^2/256$ times in the above sum (here we used the assumption $d \ge 16$).
	Therefore, $nd/2=m=|E| \le 2|E'| \le 2^{10} n^2f(2d)/d^2$ and thus $d^3 \le 2^{11}nf(2d)$.
	
	Considering (ii), we have $d\ge 50n^{1/2}$ since $m\ge 25n^{3/2}$.  Assume w.l.o.g.\ that $|B'|\ge |A'|$ 
	and therefore $|B'| \ge n/4$ and $|A'| \le n$.
	Let $\#K_{2,1}$ be the number of $K_{2,1}$'s in $G'$ with two vertices in $A'$ and 
	let $\#K_{2,2}$ be the number of $K_{2,2}$'s in $G'$.
	Then,
	$$\#K_{2,2} =  \frac{1}{4}\sum_{a \in A', b \in B', (a,b) \in E'}  |E(G'[(N_{G'}(a)\setminus \{b\})\cup (N_{G'}(b)\setminus\{a\}])| \le \frac{m}{4}f(2d) = \frac{dn}{8} f(2d).$$ 
	
	On the other hand, denoting $s_{a,a'}=|N_{G'}(a)\cap N_{G'}(a')|$ for two distinct vertices $a, a'\in A'$ we have $\#K_{2,1} = \sum_{a\ne a' \in A'} s_{a,a'}$ and:
	$$\#K_{2,2}=\sum_{a\ne a'\in A'}\binom{s_{a,a'}}{2}=\frac{1}{2}\sum_{a\ne a'\in A'}(s_{a,a'}^2-s_{a,a'}) \ge \frac{(\#K_{2,1})^2}{2\binom{|A'|}{2}}-\frac{\#K_{2,1}}{2}\ge \frac{(\#K_{2,1})^2}{2\binom{n}{2}}-\frac{\#K_{2,1}}{2},$$
	where the second to last inequality follows from Cauchy's inequality. 
	For $\#K_{2,1}$ we have the following lower bound:	
	
	$$\#K_{2,1}=\sum_{a\ne a'\in A'}s_{a,a'} = \sum_{b \in B'} \binom{\deg(b)}{2} \ge \frac{n}{4}\binom{d/8}{2}\ge nd^2/2^{11}\ge n^2,$$
	where we used double-counting by the vertex from $B'$ for the $K_{2,1}$'s and the inequalities $|B'|\ge n/4$ and $d\ge 50n^{1/2}$. Combining the inequalities above we get:
	$$\frac{dnf(2d)}{8} \ge \#K_{2,2} \ge \frac{(\#K_{2,1})^2}{n^2}-\frac{\#K_{2,1}}{2} \ge \frac{(\#K_{2,1})^2}{n^2}-\frac{(\#K_{2,1})^2}{2n^2} = \frac{(\#K_{2,1})^2}{2n^2} \ge \frac{n^2d^4}{2^{22}\cdot 2n^2}=\frac{d^4}{2^{23}}.$$	
	This implies that $d^3\le 2^{20}nf(2d)$, as claimed.
\end{proof}

Using Theorem~\ref{thm:f-sparse-neighborhoods} we can now prove Theorem~\ref{thm:spare-neighborhoods-ub} which we restate for convenience.
\KSTub*

\begin{proof}
	Recall that every graph contains a bipartite subgraph whose size is at least half of the size of the graph.	
Let $G'=(A'\cup B',E')$ be such a bipartite subgraph of $G$, where $|E'| \ge |E(G)|/2$, and let $n=|V(G)|$ and $d'=\overline{d}(G')$.
If every pair of vertices in $G$ has $f$-sparse sub-bineighborhoods, then the same holds for $G'$ and it follows from Theorem~\ref{thm:f-sparse-neighborhoods} that $d'^3 \le 2^{11} nf(2d') = O(nd'^{2-\alpha})$.
Therefore $d' \le O(n^{1/(\alpha+1)})$ and the upper bound follows.

If $0 < \alpha \le 1$ and every pair of adjacent vertices has $f$-sparse sub-bineighborhoods, then we may assume that $|E'|\ge 25n^{3/2}$ for otherwise the theorem trivially holds.
Therefore, by Theorem~\ref{thm:f-sparse-neighborhoods} we have $d'^3 \le 2^{20} nf(2d') = O(nd'^{2-\alpha})$ and $d' \le O(n^{1/(\alpha+1)})$ as before.
\end{proof}

Next we prove Theorem~\ref{thm:spare-neighborhoods-lb} by which the upper bound of Theorem~\ref{thm:spare-neighborhoods-ub} is asymptotically tight.
For convenience we restate the theorem.
\KSTlb*

\begin{proof}
	The case $\alpha = 0$ is trivial.
	For $\alpha = 1$ we can take a $K_{2,2}$-free graph with $\Theta(n^{3/2})$ edges~\cite{Kovari}. Indeed, in such a graph any bipartite subgraph induced by two sub-bineighborhoods contains no vertex of degree at least two.
	This also shows why for $\alpha > 1$ we require that every pair of vertices satisfies the $f$-sparse sub-bineighborhoods property and not just adjacent pairs. Indeed, for adjacent pairs the sub-bineighborhoods contain no edges at all in a $K_{2,2}$-free graph, so the assumption holds with arbitrary $\alpha<2$ for such neighborhoods, yet the number of edges is $\Theta(n^{3/2})$.
	
	Next we show the lower bound for $0 < \alpha < 1$.
	For convenience fix $c = 2-\alpha> 1$.
	Thus, we wish to construct a graph with $\Omega\left(n^{\frac{4-c}{3-c}}\right)$ edges in which every pair of vertices has $f$-sparse sub-bineighborhoods for $f(x)=O(x^c)$.
	Specifically, we construct and analyze a standard random bipartite graph $G=(A\cup B,E)$, where $|A|=|B|=n$ and every pair of vertices $a \in A$ and $b \in B$ are connected by an edge independently with probability $p=n^{-\frac{2-c}{3-c}}$.
	
	Let $E_{A',B'} := E(G[A'\cup B'])$ for every $A' \subseteq A$ and $B' \subseteq B$.
	%
	We say that a $4$-tuple $a \in A, b \in B, A' \subseteq A$ and $B' \subseteq B$ is `bad' if
	$B' \subseteq N(a) \setminus \{b\}$, $A' \subseteq N(b)\setminus \{a\}$ and $|E_{A',B'}| > q(|A'|+|B'|)^{c}$, where $q>0$ is a large enough constant that will be determined later.
	Such a fixed $4$-tuple is bad with probability $p^{|A'|}p^{|B'|} \cdot \Pr(|E_{A',B'}| > q(|A'|+|B'|)^{c}$.
	Let $X$ be a random variable that is equal to the number of bad $4$-tuples.
	Then,  	
	$$\E{X} = n^2 \sum_{i=1}^{n-1} \sum_{j=1}^{n-1} {{n-1}\choose{i}} p^i {{n-1}\choose{j}} p^j \cdot \underset{Z \sim \textrm{Bin}(ij,p)}{\Pr}\left(Z > q(i+j)^{c}\right),$$ where $Z$ is a binomially distributed random variable counting the number of `successes' among $i\cdot j$ independent experiments each of success probability $p$.
	It is enough to show that $\E{X} \rightarrow 0$ as $n \rightarrow \infty$, since by Markov's inequality $\Pr(X \ge 1) \le \E{X}$, which means that $\E{X} \rightarrow 0$ implies that there are no bad $4$-tuples with high probability.
	Furthermore, it follows from Chernoff bound that with high probability $|E|=\Theta(\E{E}) = \Theta(pn^2) = \Theta(n^{(4-c)/(3-c)})$, therefore, with high probability there are $\Theta(n^{(4-c)/(3-c)})$ edges and no bad $4$-tuples.
	
	Let
	$$ T_{i,j} := \binom{n}{i} p^i \binom{n}{j} p^j \cdot \underset{Z_{i,j} \sim \textrm{Bin}(ij,p)}{\Pr}(Z_{i,j} > q(i+j)^{c}) .$$
	
	Since $\E{X} \le n^4 \cdot \underset{i,j}{\max} \{T_{i,j}\}$, it is enough to show that the maximum value of $T_{i,j}$ tends to zero faster than $n^{-4}$ as $n \rightarrow \infty$.
	Note that $ \underset{i,j}{\max} \{T_{i,j}\}$ is obtained when $i$ and $j$ are equal or almost equal in case $i+j$ is odd.
	Indeed, suppose that the sum $s=i+j$ is fixed.
	Then clearly $\underset{Z_{i,j} \sim \textrm{Bin}(ij,p)}{\Pr}(Z_{i,j} > q(i+j)^{c})$ is maximized when $ij$ is maximized and this happens for $i=\lfloor s/2 \rfloor$ and $j=\lceil s/2 \rceil$ (or the other way around).
	Likewise, it is not hard to show that the product of the binomial coefficients  $\binom{n}{i}\binom{n}{j}=\binom{n}{i}\binom{n}{s-i}$ is also maximized for $i=\lfloor s/2 \rfloor$ and $j=\lceil s/2 \rceil$, see~\cite{stack}. 
	Thus, we will consider $\max_i \{T_{i,i}, T_{i,i+1}\}$ 
	and split the analysis into three cases.
	
	Let $A_{i,j} := \binom{n}{i}p^i \binom{n}{j}p^j$ be the first term of $T_{i,j}$ and let $B_{i,j} := \underset{Z_{i,j} \sim \textrm{Bin}(ij,p)}{\Pr}(Z_{i,j} > q(i+j)^{c})$ be its second term.
	Note that $\frac{A_{i,i+1}}{A_{i,i}} = p\frac{n-i}{i+1} \le n$.	
	Using the inequality $\binom{n}{i} \le (ne/i)^i$ we get that
	\begin{equation}\label{eq:binom-coef}
		A_{i,i} = \left(\binom{n}{i} p^{i}\right)^2 \le \left(\frac{e\cdot n^{1/(3-c)}}{i}\right)^{2i} \textrm{ and thus } A_{i,i+1} \le n\left(\frac{e\cdot n^{1/(3-c)}}{i}\right)^{2i}.
	\end{equation}
	
	\paragraph{Case 1: $i > 10e\cdot n^{1/(3-c)}$.}
	It follows from Eq.~\eqref{eq:binom-coef} that $T_{i,i}, T_{i,i+1} < n 0.1^{2i} \le n 0.1^{20en^{1/(3-c)}}$. As $n \rightarrow \infty$ this expression decays to zero much faster than $n^{-4}$ and that concludes this case.
	
	\paragraph{Case 2: $i \le 10e\cdot n^{\frac{1}{3-c}}$.}
	Let $\mu_{i,j} = \E{Z_{i,j}} = ijp$ and let $t_{i,j} = q(i+j)^c$.
	Note that $2qi^c \le t_{i,i}, t_{i,i+1} \le 9qi^c$, $\mu_{i,i}=i^2 p$ and $i^2 p \le \mu_{i,i+1} \le 2i^2 p$. 
	Recall that by Chernoff bound for any $\delta \ge 0$ we have 
	$$\Pr(Z_{i,j}>(1+\delta)\mu_{i,j}) \le \left(\frac{e^{\delta}}{(1+\delta)^{1+\delta}}\right)^{\mu_{i,j}} = 
	e^{\mu_{i,j}(\delta-(1+\delta)\ln(1+\delta))}.$$
	
 Thus,
	$$\ln B_{i,j} \le \mu_{i,j}(\delta - (1+\delta)\ln(1+\delta)).$$

	Note that if we choose $q \ge (10e)^{2-c}$ , then  $\frac{t_{i,i}}{\mu_{i,i}}, \frac{t_{i,i+1}}{\mu_{i,i+1}} \ge \frac{q}{p}i^{c-2} \ge 1$.
	Substituting $\delta=\frac{t_{i,j}}{\mu_{i,j}}-1 \ge 0$ for $j=i,i+1$ we get:
	\begin{align*}
		\ln B_{i,j} & \le t_{i,j} - \mu_{i,j} - t_{i,j}(\ln t_{i,j} - \ln \mu_{i,j}) =  t_{i,j}(1- \ln t_{i,j} + \ln \mu_{i,j}) -\mu_{i,j} \\
		& \le 9qi^c \left(1-\ln(2q) - c\ln i + \ln 2 + 2\ln i + \ln p\right) \\
		& \le 9qi^c \left(2-\ln(2q)+(2-c)\ln i - \frac{2-c}{3-c}\ln n\right). 
	\end{align*}
	
	Next we split the current case into two subcases and show that $\ln T_{i,i}$ and $\ln T_{i,i+1}$
	tend to $-\infty$ fast enough in both of them. 
	
	\paragraph{Subcase 2a: $n^{\frac{1}{2(3-c)}} \le i \le 10e\cdot n^{\frac{1}{3-c}}$.}
	Considering the natural logarithm of $T_{i,j}$ for $j=i,i+1$ we have:
	\begin{align*} 
		\ln T_{i,j} & = \ln A_{i,j} + \ln B_{i,j} \\
		& \le \ln n + 2i \left(1+\frac{\ln n}{3-c} - \ln i\right) + 9qi^c \left(2-\ln(2q)+(2-c)\ln\left(10e\cdot n^\frac{1}{3-c}\right)  - \frac{2-c}{3-c}\ln n\right) \\
		& = \ln n + 2i \left(1+\frac{\ln n}{3-c} - \ln i\right) + 9qi^c \left(2-\ln(2q)
		+ (2-c)\ln(10e) \right)\\
		& \le 3i \ln n + 9qi^c(9-\ln(2q)).
	\end{align*}
	
	Therefore, if we take $q>0.5e^9$ (which we do), then since $i\ln n \ll i^{c}$ when $c>1$ and $i \ge n^{\frac{1}{2(3-c)}}$, we conclude that $T_{i,i}$ and $T_{i,i+1}$ decay to zero super-polynomially in this case.
		
	\paragraph{Subcase 2b: $i < n^{\frac{1}{2(3-c)}}$.} Recall that for $j=i,i+1$ we have:
	\begin{align*} 
		\ln T_{i,j} & \le \ln n + 2i \left(1+\frac{\ln n}{3-c} - \ln i\right) + 9qi^c \left(2-\ln(2q)+(2-c)\ln i - \frac{2-c}{3-c}\ln n\right) \\ 
		& \le \ln n + 2i \left(1+\frac{\ln n}{3-c} - \ln i\right) + 9qi^c \left(2-\ln(2q)+\frac{2-c}{2(3-c)}\ln n - \frac{2-c}{3-c}\ln n\right) \\ 
		& = \ln n + 2i \left(1+\frac{\ln n}{3-c} - \ln i\right) + 9qi^c\left(2- \ln (2q) - \frac{\ln n}{2(3-c)} \right) \\ 
		& = \left(1+ \frac{2i}{3-c} - \frac{9qi^c}{2(3-c)}\right)\ln n 
		-2i(\ln i - 1) - 9qi^c(\ln(2q)-2) \\ 
		& < \frac{2(3-c)+4i-9qi^c}{2(3-c)}  \ln n < \frac{-qi^c}{2(3-c)}\ln n < -5\ln n,
	\end{align*}
	where the last two inequalities hold for large enough $q$.
	Therefore, $n^4 T_{i,i}$ and $n^4 T_{i,i+1}$ tend to zero also in this case.
	Considering the conditions on $q$ along the proof, it is enough to take, say, $q=5000$.
\end{proof}

\paragraph{Remarks:}
\begin{itemize}
	\item
	As mentioned in the Introduction, Theorem~\ref{thm:spare-neighborhoods-ub} implies the theorem of Erd\H{o}s and Simonovits (Theorem~\ref{thm:erdos-simonovits}) which in turn implies the K\H{o}v\'ari-S\'os-Tur\'an Theorem.
	Considering the differences between Theorem~\ref{thm:erdos-simonovits} and Theorem~\ref{thm:spare-neighborhoods-ub}, firstly, the former can be used only for sub-bineighborhoods which induce a sparse subgraph due to avoiding a certain subgraph, whereas in the latter we just assume sparseness. Secondly, Theorem~\ref{thm:spare-neighborhoods-ub} can be used for very sparse induced subgraphs (with sub-linearly many edges when $\alpha > 1$) as long as the condition in the theorem holds for every pair of vertices (rather than for adjacent ones).
	And thirdly, for Theorem~\ref{thm:spare-neighborhoods-ub} we have a matching lower bound for every $0 \le \alpha \le 1$ (Theorem~\ref{thm:spare-neighborhoods-lb}), while for Theorem~\ref{thm:erdos-simonovits} a matching lower bound is only known for $\alpha=1,\frac{1}{2}$, and any tight bound for some $\alpha=\frac{1}{t}$, $t>2$, would be a remarkable result.
	
	\item 
	The proof of Theorem~\ref{thm:spare-neighborhoods-ub} is quite similar to the proof of Theorem~\ref{thm:erdos-simonovits} in~\cite{ES70}.
	In both cases, one first makes the graph close to being regular and then uses double-counting, however, in the case of Theorem~\ref{thm:erdos-simonovits} the first part requires more effort. 
	Chronologically, we first proved the first part of Theorem~\ref{thm:spare-neighborhoods-ub} which was needed for the proof of Theorem~\ref{thm:1-int}, only later did we learn about Theorem~\ref{thm:erdos-simonovits} and used it to prove the second part of Theorem~\ref{thm:spare-neighborhoods-ub}.
	
	\item
	It would be interesting to determine whether Theorem~\ref{thm:spare-neighborhoods-ub} holds when one assumes sparsity of the subgraph induced by the bineighborhood of every pair of (adjacent) vertices instead of assuming sparsity of subgraphs induced by each of their sub-bineighborhoods.
	Note that the first part of the proof of Theorem~\ref{thm:spare-neighborhoods-ub} does not go through in such a case. 
	Clearly, $f$-sparse bineighborhoods do not imply $f$-sparse sub-bineighborhoods.
	For example, recall that $(n^{3/2})/4 \le \mbox{ex}(n,K_{2,2}) \le n^{3/2}$ for sufficiently large $n$~\cite{Kovari} and consider a $K_{2,2}$-free bipartite graph $G$ with $n$ vertices on each side and $\Theta(n^{3/2})$ edges. 
	Assume further that the degree of every vertex in $G$ is at least $c\sqrt{n}$ and at most $C\sqrt{n}$, for some constants $0<c<C$.
	Note that every pair of vertices in $G$ has an $f$-sparse bineighborhood for $f(x)=x/2$.
	Now add $c'=20$ new vertices to each side and connect them to all the vertices on the other side.
	Then for $g(x)=x^{1.6}$ and large enough $n$ the following holds.
	For two original vertices, their bineighborhood contains at least $2c' + 2c\sqrt{n} -2\ge 2c\sqrt{n}$ vertices which induce a subgraph with at most $C\sqrt{n}+2c'C\sqrt{n} + c'^2\le g(2c\sqrt{n})$ edges;
	the bineighborhood of one original vertex and one new vertex consists of at least $n+2c'+c\sqrt{n}-2\ge n$ vertices and it induces a subgraph with at most $(n+C\sqrt{n})^{3/2}+c'n+c'^2+c'C\sqrt{n} \le g(n)$ edges;
	and the bineighborhood of two new vertices consists of at least $2n+2c'-2\ge 2n$ vertices and it induces a subgraph with at most $(2n)^{3/2} + 2c'n +c'^2 \le g(2n)$ edges.
	Thus, in the resulting graph every pair of vertices has a $g$-sparse bineighborhood.
	However, the subgraph induced by the new vertices is a sub-bineighborhood of any two original vertices from different sides and this subgraph has $2c'$ vertices and $c'^2=400$ edges which is more than $g(2c')$.

	\item 
	Theorem~\ref{thm:spare-neighborhoods-ub} might find other applications, e.g., in other scenarios where it can be combined with Theorem~\ref{thm:intersection-reverse}. 
	Here is one example. For a graph $H$ let $\ex_{\textrm{cr}}(n,H)$ be the maximum number of edges in a topological graph on $n$ vertices with no self-intersecting copy $H$. Note that $\ex(n,H)\le \ex_{\textrm{cr}}(n,H)$ and if $H$ is non-planar then $\ex(n,H)=\ex_{\textrm{cr}}(n,H)$. 
	In \cite{janzer2025} Theorem~\ref{thm:intersection-reverse} is used to show that $\ex_{\textrm{cr}}(n,C_4)=O(n^{3/2})$ (and this is tight). Using Theorem~\ref{thm:spare-neighborhoods-ub} this immediately implies:
	\begin{corollary}
		$\ex_{\textrm{cr}}(K_{2,3})=O(n^{5/3})$ and $\ex_{\textrm{cr}}(n,K_{3,3}^-)=O(n^{5/3})$ where $K_{3,3}^-$ is the complete bipartite graph on $3+3$ vertices minus one edge.
	\end{corollary}
\end{itemize}

\paragraph{Acknowledgments.}
We are grateful to Rom Pinchasi for many helpful discussions on the problems studied in this paper. 
In particular he has suggested most of the ingredients of the proofs of Theorems~\ref{thm:bi-infinite} and~\ref{thm:x-mon} and later found out that similar results were already proved by Pach and Sharir~\cite{Pach-Sharir}.
We thank Bal\'azs Patk\'os for his comments about Theorem \ref{thm:spare-neighborhoods-ub}.
For the proof of Theorem~\ref{thm:spare-neighborhoods-lb} we used some back and forth interaction with Google's Large Language Model `Gemini' (\texttt{https://gemini.google.com/}).

\footnotesize
\bibliographystyle{plainurl}
\bibliography{bibliography.bib}

\end{document}